\documentclass[final]{siamart190516}

\usepackage{amssymb,amsmath,amsfonts,mathtools,enumerate,color}

\usepackage[toc,page,title,titletoc,header]{appendix}
\usepackage{graphicx,subfig}

\usepackage{paralist}

\usepackage{booktabs}

\usepackage{tikz}
\usetikzlibrary{arrows,positioning,shapes.geometric}

\newtheorem{Theorem}{Theorem}[section]
\newtheorem{Lemma}[Theorem]{Lemma}
\newtheorem{Proposition}[Theorem]{Proposition}
\newtheorem{Assumption}{Assumption}

\theoremstyle{definition}
\newtheorem{Definition}{Definition}[section]
\newtheorem{Problem}{Problem}[section]

\theoremstyle{remark}
\newtheorem{Remark}{Remark}[section]

\def \Vh0{\stackrel{\circ}{V}_h} \def\to{\rightarrow}
   
\def\Om{\Omega}   
 
\newcommand{\q}{\quad}

\def\l{\label}  \def\f{\frac}  \def\fa{\forall}
\def\b{\beta}  \def\a{\alpha} 
 \def \th{\theta}
\def\eps{\varepsilon}

 \def\t{\times}

\def\u{{\bf u}} \def\v{{\bf v}}

\def\cA{\mathcal{A}}
\def\cB{\mathcal{B}}

\def\cF{\mathcal{F}}

\def\cI{\mathcal{I}}

\def\cL{\mathcal{L}}
\def\cM{\mathcal{M}}
\def\cN{\mathcal{N}}

\def\cO{\mathcal{O}}

\def\bA{{\textbf{A}}}

\def\N{{\mathbb{N}}}
\def\bP{\mathbb{P}}

\def\R{{\mathbb R}}

\newcommand{\ex}{\mathbb{E}}

\def\bb{\begin{equation}} \def\ee{\end{equation}}
\def\bn{\begin{enumerate}} \def\en{\end{enumerate}}

\begin{document}

\title{A penalty scheme for monotone systems with interconnected obstacles: convergence and error estimates}
\author{
Christoph Reisinger\thanks{Mathematical Institute, University of Oxford, United Kingdom ({\tt christoph.reisinger@maths.ox.ac.uk, yufei.zhang@maths.ox.ac.uk})}
\and
Yufei Zhang\footnotemark[2]
}

\maketitle

\begin{abstract}
  We present a novel penalty approach for a class of quasi-variational inequalities (QVIs) involving monotone systems and interconnected obstacles. We show that for any given positive switching cost, the solutions of the penalized equations 
converge monotonically to those of the QVIs. We estimate the penalization errors and are able to deduce that 
the optimal switching regions are constructed exactly. 
We further demonstrate that as the switching cost tends to zero, the QVI degenerates into an  equation of HJB type,  which is approximated by the penalized equation at the same order (up to a log factor) as that for positive switching cost.
Numerical experiments on optimal switching problems  are presented to illustrate the theoretical
results and to demonstrate the effectiveness of the method.

\end{abstract}



\section{Introduction}\l{sec:introduction}

In this article, we consider the following discrete quasi-variational inequality (QVI): 
\begin{Problem} \l{pro:switching}
Find $u=(u^1,\ldots, u^d)\in \R^{N\t d}$,  such that
\bb\l{eq:switching}
G_i(u)=\min (F_i(u),\; u^i-\cM_{i} u)=0,\q i\in \cI\coloneqq \{1,\cdots, d\},\, d\ge 2,
\ee
where 
$\cM_{i} u\coloneqq \max_{j\not =i}(u^j-c^{i,j})$ with $c^{i,j}\in (0,\infty)$ for all $j\not=i$, and $F=(F_i)_{i\in\cI}:\R^{N\t d}\to \R^{N\t d}$ is a continuous function satisfying the following  condition:
\begin{itemize}
\item 
There exists a constant $\gamma>0$ such that for any  $u,v\in \R^{N\t d}$, $i\in\cI$ and $l\in \cN\coloneqq \{1,\ldots, N\}$ with $u^i_l-v^i_l=\max_{j\in \cI,k\in\cN}u^j_k-v^j_k\ge 0$, we have
\bb\l{eq:mono}
F_i(u)_l-F_i(v)_l\ge \gamma(u^i_l-v^i_l).
\ee
\end{itemize}
\end{Problem}

Following \cite{djehiche2017}, we will refer to $\cM u$ as an interconnected obstacle, where $\cM$ is a special form of the intervention operator in \cite{azimzadeh2016weakly}. Similar to \cite{briani2012} in the continuous context, we shall also refer to a continuous function $F=(F_i)_{i\in\cI}$ with condition \eqref{eq:mono} as a monotone system.
Such discrete systems arise naturally from  a sensible discretization (i.e., stable and convergent)  of elliptic and parabolic QVIs associated to hybrid optimal control problems with switching controls (see e.g.~\cite{pham2009,azimzadeh2016weakly,djehiche2017,ferretti2017}). Since each $F_i$ could depend on all components of the solution $u$, the system $F$ in Problem \ref{pro:switching} not only includes systems of Isaacs equations  \cite{jakobsen2006} or possibly nonlocal Hamilton-Jacobi-Bellman (HJB) variational inequalities \cite{witte2012,reisinger2018}, but also coupled systems stemming from regime switching models \cite{babbin2014} and non-coorperative games \cite{dockner2000}. To simplify the presentation, we shall focus on the case with constant switching cost 
$c^{i,j}\equiv c>0$ for all $j\not =i$, $i\in \cI$, but the results and analysis extend naturally to the  cases with general switching costs $c^{i,j}_l>0$ for all $j\not =i$.

By rewriting the  interconnected obstacle $u^i-\cM_{i} u$ as $\min_{j\not =i}(P^ju-c)$, one can easily see that  the matrices $(P^j)_{j\not=i}$  involved are neither $M$-matrices nor weakly chained diagonally dominant matrices (\cite{azimzadeh2016weakly}). Hence a direct application of policy iteration could fail due to the  possible singularity of the matrix iterates. However, 
as we shall see later, penalty schemes are always applicable (even for zero switching cost) and it is straightforward to derive convergent iterative solvers for the penalized equations, 
 which make penalty schemes more appealing for solving the QVI  \eqref{eq:switching}. Thus it is important to design efficient penalty schemes for general monotone systems with interconnected obstacles.

Moreover, the implementation of penalty schemes, in particular the choice of penalty parameters, depends greatly on the accuracy of the penalty approximation with   a given penalty parameter, hence it is practically important to quantify the penalty errors.
However, the non-diagonal dominance of the  matrices $(P^j)_{j\not=i}$ poses a significant challenge for estimating the penalization errors. In fact, an essential step in estimating the penalty error for standard obstacle problems is to show that if $u^\rho$ solves the penalized equation with the penalty parameter $\rho\ge 0$, then $u^\rho+C/\rho$ is a feasible solution to the obstacle problem (i.e., it satisfies the constraint) for a large enough constant $C$ independent of the penalty parameter (see e.g.~\cite{reisinger2018}). But this is clearly false in the current setting since the interconnected obstacles remain invariant under any vertical shift of the solutions. We shall overcome these difficulties by introducing certain regularization procedures, which consist of approximating the switching problem by a sequence of  obstacle problems involving diagonally dominant  matrices, and recover the same convergence rates  (up to a log factor) as those for conventional obstacle problems.  

Finally, we remark that the monotonicity condition \eqref{eq:mono} is essentially different from the monotonicity in the Euclidean norm discussed in \cite{huang2010},  which enables us to consider penalty methods for the QVIs of fully nonlinear degenerate equations including Isaacs equations. To  the best of our knowledge, there is no published penalty scheme with  rigorous error estimates covering such general QVIs. However, there is a vast literature on penalty methods for variational inequalities (see e.g.~\cite{forsyth2002,ito2006,huang2010}).
For  works covering specific extensions, we refer the reader to \cite{witte2011} for penalty approximations to HJB equations, to \cite{witte2012,reisinger2018} for applying policy iteration together with penalization to solve HJB variational inequalities, and to \cite{bensoussan1982,azimzadeh2016weakly,azimazadeh2018} for an application of penalty schemes to classical HJBQVIs (without error estimates).

The main contributions of our paper are:
\begin{itemize} 

\item  
We propose penalty schemes for discrete monotone systems with interconnected obstacles.
We present a novel analysis technique for the well-posedness of the penalized equations with a general class of penalty terms by smoothing the monotone systems. We further demonstrate that the solution of the penalized equation converges to the solution of \eqref{eq:switching} monotonically from below, which  subsequently gives a
constructive proof for the existence of  solutions to Problem \ref{pro:switching}.

\item 
Based on regularizations of the interconnected obstacles, we  estimate the penalization error for monotone systems with concave nonlinearity, which include HJBQVIs as special cases (see the discussions below Assumption \ref{assum:concave}). We introduce two iterative regularization procedures, namely the iterated  optimal stopping approximation and the time-marching iteration, which enable us to demonstrate that for any given positive switching cost, the penalty approximation using a penalty term with degree $\sigma>0$ enjoys convergence of order $\cO(\rho^{-\sigma}
\ln\rho)$ as  the penalty parameter $\rho\to \infty$, independent of the number of switching regimes $d$. 
We emphasize that, unlike the error estimates for HJBQVIs in \cite{boulbrachene2001fem,ferretti2017}, our analysis does not require the running reward functions  or the solutions to have a unique sign.
Moreover, our error estimate also enables us to exactly construct the switching regions of Problem \ref{pro:switching}, which to the best of our knowledge is new, even for classical HJBQVIs.

\item We further investigate the limiting case with zero switching cost, where Problem \ref{pro:switching} degenerates to an equation of  HJB type, i.e., Problem \ref{pro:hjb} below. In this case, the penalty scheme of \eqref{eq:switching} leads to a novel penalty scheme for  such equations, which admits the same convergence rates (up to a log factor) as those for fixed positive switching cost when the penalty parameter tends to infinity. 
We remark that this error estimate applies to non-convex/non-concave systems, such as systems of Isaacs equations.

\item Contrary to \cite{azimzadeh2016weakly,azimazadeh2018}, the penalty is applied to each component of the system, which enables us to derive easily implementable and efficient iterative schemes for penalized equations  without taking the pointwise maximum over all switching components.

\item Numerical examples for  infinite-horizon optimal switching problems in the two-regime case and the three-regime case are included to   illustrate the theoratical results for the asymptotic  behaviours of the penalty errors with respect to the penalty parameter and  switching cost.
\end{itemize}

We now summarize some of our main results  in the following diagram. Suppose the function $F$ in Problem \ref{pro:switching} is concave and the penalty function in Problem \ref{pro:penalty} below is given by $\pi(y)=y^+$. Then 
the following error estimates hold:

$$
\begin{tikzpicture}[>=latex']
        \tikzset{block/.style= {draw, rectangle, align=center,minimum width=2cm,minimum height=1cm},
        }
        \node [block]  (urc) {$u^{c,\rho}$ solves Problem \ref{pro:penalty}\\ with $c,\rho>0$};

        \node [block, right = 1cm of urc] (urc_r){$u^{c,\rho}\nearrow u^c$ at rate \\ 
        $\cO\left(\ln\rho/(c \rho)\right)$
        as $\rho\to \infty$};
        \node [block, right = 1cm of urc_r] (uc){$u^{c}$ solves Problem \ref{pro:switching}\\ with $c>0$};
	
	\node [block, below = 0.6cm of urc] (urc_c){$u^{c,\rho}\nearrow u^\rho$ at  rate \\ $\cO(c\rho)$ as $c\to0$};
        \node [block, below = 0.6cm of urc_c] (ur){$u^{\rho}$ solves Problem \ref{pro:penalty}\\ with $c=0$, $\rho>0$};
	
	\node [block, right = 1.65cm of ur] (ur_r){$u^{\rho}\nearrow \u$ at  rate \\ $\cO(1/\rho)$ as $\rho\to \infty$};

	\node [block, below = 0.6cm of uc] (uc_c){$u^{c}\nearrow \u$ at  rate \\ $\cO(c)$ as $c\to0$};
	\node [block, below = 0.6cm of uc_c] (u){$\u$ solves Problem \ref{pro:hjb}};
	
        \path[draw, ->]
            (urc) edge (urc_r)
            (urc_r) edge (uc)
            (urc) edge (urc_c)
            (urc_c) edge (ur)            
            (ur) edge (ur_r)
            (ur_r) edge (u)         
            (uc) edge (uc_c)         
            (uc_c) edge (u)            
                       ;
    \end{tikzpicture}
$$

The remainder of this paper is organized as follows. We shall propose  a  class of penalty approximations to Problem \ref{pro:switching} in Section \ref{sec:penalty},  and demonstrate its well-posedness and monotone convergence. Then we construct two regularization procedures for Problem \ref{pro:switching} in Section \ref{sec:regularization}, which enable us to obtain an error estimate of the penalization error for  QVIs with positive switching cost in Section \ref{subsec:error_positive}. 
We then proceed to estimate the penalization errors for QVIs with vanishing switching cost in Section \ref{sec:error_zero}.
Numerical examples for  two-regime and three-regime optimal switching problems are presented in Section \ref{sec:numerical} to illustrate the effectiveness of our algorithms.

\section{Penalty approximations of  QVIs}\l{sec:penalty}
In this section, we discuss how Problem \ref{pro:switching} can be approximated by a sequence of penalized equations. The well-posedness of the penalized equations and  their  monotone convergence shall be established, which subsequently lead to a constructive proof for the well-posedness of \eqref{eq:switching}.  

We start by collecting some useful notation. 
For any given matrix $A, B\in \R^{d_1\t d_2}$, 
we denote by $A\ge B$ the relation $A_{ij} \ge B_{ij}$ for all indices $i,j$, by $\min(A,B)$ the matrix of elements $\min(A_{ij},B_{ij})$, 
by  $A^+=\max(A,0)$ (resp.~$A^-=\max(-A,0)$) the (element-wise) positive  (resp.~negative) part of $A$, and by $\|A\|$ the  usual sup-norm $\| A\|=\max_{i,j}|A_{ij}|$.

Before introducing the penalty equations, we first adapt the non-loop arguments  in \cite{jakobsen2006} to the current discrete setting and establish a comparison theorem of the QVI \eqref{eq:switching}.

\begin{Proposition}\l{prop:comp_s}
Suppose $c>0$ and $u=(u^i)_{i\in \cI}$ (resp.~$v=(v^i)_{i\in \cI}$) satisfies 
$$
\min (F_i(u),\; u^i-\cM_{i} u)\le 0\q \textnormal{(resp.~$\ge 0$)}, \q i\in \cI,
$$
then we have $u\le v$.
\end{Proposition}

\begin{proof}
%
%
%

Let $M\coloneqq \max_{j,k}(u^j_k-v^j_k)=u^i_l-v^i_l$. Suppose that $u^{j}_l\le (M_ju)_l$ for all $j\in \Gamma$, where $\Gamma\coloneqq \{j\mid u^j_l-v^j_l=M\}$. Pick $i_1\in \Gamma$ such that $u^{i_1}_l\le (M_{i_1}u)_l=u^{i_2}_l-c$ for some $i_2\not=i_1$. Since $v^{i_1}_l\ge (\cM_{i_1}v)_l\ge v^{i_2}_l-c$, we have 
$$
u^{i_1}_l-u^{i_2}_l\le -c\le v^{i_1}_l-v^{i_2}_l,
$$
which implies $u^{i_2}_l-v^{i_2}_l\ge u^{i_1}_l-v^{i_1}_l=M$. By the maximality of $M$, the previous inequality is in fact an equality and hence $i_2$ is in $\Gamma$. Continuing this way, we can pick indices $i_3,i_4,\ldots$ in $\Gamma$ such that $u^{i_k}_l-u^{i_{k+1}}_l\le -c$ for all $k\ge 1$, which further implies that
$$
u^{i_1}_l-u^{i_n}_l=\sum_{k=1}^{n-1} u^{i_k}_l-u^{i_{k+1}}_l\le -(n-1)c<0,\q \textnormal{for $n>1$.}
$$
Since $\Gamma$ is finite, we can use the pigeonhole principle to find $n>1$ with $i_1-i_n=0$, arriving at a contradiction.

The above argument establishes that $u^{i_0}_l> (M_{i_0}u)_l$ for some $i_0\in \Gamma$. Consequently, we have $F_{i_0}(u)_l\le 0\le  F_{i_0}(v)_l$. Combining this with the monotonicity \eqref{eq:mono} of $F$, we have $M\le 0$.
\end{proof}

A direct consequence of Proposition \ref{prop:comp_s} is the uniqueness  of solutions to \eqref{eq:switching}. The existence of solutions to Problem  \ref{pro:switching} shall be established constructively via penalty approximations below (see Remark \ref{rmk:wp}).

Now we are ready to propose the penalty approximation of the QVI \eqref{eq:switching}, which is  an extension of the ideas used for  HJB obstacle problems in \cite{witte2012,reisinger2018}.
For any given parameter $\rho \ge 0$, we shall consider the following penalized problem: 
\begin{Problem}\l{pro:penalty}
Find $u^\rho=(u^{\rho,i})_{i\in\cI}\in \R^{N\t d}$ such that 
\bb\l{eq:penalty}
G^\rho_i(u^\rho)_l\coloneqq F_i(u^{\rho})_l-\rho\sum_{j\not =i}\pi( u^{\rho,j}_l-c-u^{\rho,i}_l)=0,\q i\in \cI, \, l\in\cN,
\ee
where the penalty term $\pi:\R\to \R$ is a continuous non-decreasing function satisfying $\pi|_{(-\infty,0]}=0$ and $\pi|_{(0,\infty)}>0$.
\end{Problem}
\begin{Remark}
In \eqref{eq:penalty}, the penalty is applied to each component of the switching system, thus  efficient iterative schemes for penalized equations  can be implemented without taking the pointwise maximum over all switching components at each index $l\in \cN$.  
However, all the statements below can be shown to hold for penalized problems with penalty terms involving the maximum of all switching components, such as $\max_{j\not =i}\pi( u^{\rho,j}-c-u^{\rho,i})$ or $\pi( \cM_iu^\rho-u^{\rho,i})$.


Due to the fact that the control $j$ takes only $d$ distinct values, we can apply the penalty term finitely many times (once per value). This is not directly possible in the framework of \cite{azimazadeh2018,azimzadeh2016weakly}, where the number of attainable values for the control in the intervention operator grows unbounded as the meshing parameter in the approximation of an infinite control set approaches zero (see \cite{kharroubi2010} for an extension of such a penalty scheme to general intervention operators with an infinite number of control values: the summation is replaced by an integral, which might subsequently be approximated by quadrature).

\end{Remark}

The following result shows a comparison principle for the penalized equation with a fixed penalty parameter $\rho$, which not only implies the uniqueness of solutions to the penalized equations, but also plays a crucial role in the  convergence analysis of the penalty approximations.

\begin{Proposition}\l{prop:comparison_p}
For any given penalty parameter $\rho\ge 0$ and switching cost $c\ge 0$, suppose $u^\rho=(u^{\rho,i})_{i\in \cI}$ (resp.~$v^\rho=(v^{\rho,i})_{i\in \cI}$) satisfies 
$$
F_i(u^\rho)-\rho\sum_{j\not =i}\pi(u^{\rho,j} -c-u^{\rho,i})\le 0 \q \textnormal{(resp.~$\ge 0$)}, \q i\in \cI,
$$
then we have $u^{\rho}\le v^{\rho}$.
\end{Proposition}

\begin{proof}

Let $M\coloneqq \max_{j,k}u^{\rho,j}_k-v^{\rho,j}_k=u^{\rho,i}_l-v^{\rho,i}_l$. Then we have
 $u^{\rho,j}_l-c-u^{\rho,i}_l\le v^{\rho,j}_l-c-v^{\rho,i}_l$ for all $j\not =i$, and hence
$\sum_{j\not =i}\pi( u^{\rho,j}_l-c-u^{\rho,i}_l)\le \sum_{j\not =i}\pi ( v^{\rho,j}_l-c-v^{\rho,i}_l)$.
This, along with the fact that 
\bb\l{eq:mono_p}
F_i(u^\rho)_l-F_i(v^\rho)_l-\rho\bigg(\sum_{j\not =i}\pi( u^{\rho,j}_l-c-u^{\rho,i}_l)-\sum_{j\not =i}\pi ( v^{\rho,j}_l-c-v^{\rho,i}_l)\bigg)\le 0,
\ee
leads to $F_i(u^\rho)_l-F_i(v^\rho)_l\le 0$. Then we can conclude from the monotonicity of $F$ that $M\le 0$.
\end{proof}

The next result presents an a priori estimate of the solution to the penalized equations, independent of the penalty parameter $\rho$ and switching cost $c$.
\begin{Lemma}\l{lem:bound_p}
Suppose $u^\rho$ solves Problem \ref{pro:penalty} with  given penalty parameter $\rho\ge 0$ and switching cost $c\ge 0$, then $\|u^{\rho}\|\le  \|F(0)\|/\gamma$.
\end{Lemma}
\begin{proof}
Let  $|u^{\rho,i}_l|=\|u^{\rho}\|$. Suppose that $u^{\rho,i}_l\ge 0$, then $u^{\rho,j}_l-c-u^{\rho,i}_l\le 0$ for all $j\not=i$, hence we deduce from \eqref{eq:mono} that 
$$
\gamma (u^{\rho}_l-0)\le  F_i(u^{\rho})_l-F_i(0)_l= \rho\sum_{j\not =i}\pi( u^{\rho,j}_l-c-u^{\rho,i}_l)-F_i(0)_l= -F_i(0)_l,
$$
thus 
$\|u^{\rho}\|\le  \|F(0)\|/\gamma$. On the other hand, suppose that $u^{\rho,i}_l<0$, we can obtain directly from  \eqref{eq:mono} and the non-negativity of $\pi$ that
$$
\gamma (0-u^{\rho}_l)\le F_i(0)_l-F_i(u^{\rho})_l\le F_i(0)_l,
$$
hence $\|u^{\rho}\|\le  \|F(0)\|/\gamma$, which leads us to the desired estimate.
\end{proof}

Now we are ready to conclude the well-posedness of the penalized problem \eqref{eq:penalty}. The following lemma has been proved in \cite[Theorem 5.3.9]{ortega2000}, which is of crucial importance for the existence of solutions to the penalized equations.

\begin{Lemma}\l{lemma:homeo}
Suppose that $F:\R^n\to \R^n$ is continuously differentiable on $\R^n$, and $\nabla F(x)$ is nonsingular for all $x\in \R^n$. Then $F$ is  a homeomorphism from $\R^n$ onto $\R^n$ if and only if $\lim_{\|x\|\to \infty}\|F(x)\|=\infty$.
\end{Lemma}

\begin{Theorem}\l{thm:wp_p}
For any given penalty parameter $\rho\ge 0$ and switching cost $c\ge 0$, Problem \ref{pro:penalty} admits a unique solution $u^{\rho}$ satisfying $\|u^{\rho}\|\le  \|F(0)\|/\gamma$.
\end{Theorem}
\begin{proof}
The uniqueness and the a priori bound have been established in Proposition \ref{prop:comparison_p} and Lemma \ref{lem:bound_p} respectively. Now we shall  prove the existence of solutions by approximating the penalized equation $G^\rho$ with a sequence of smooth equations.

Consider a family of smooth functions $\delta_m:\R^{N\t d}\to (0,\infty)$  supported in $B(0,1/m)$ with unit mass,  we define the smooth functions $G^{\rho,m}\coloneqq G^\rho \ast \delta_m$, where the convolution is applied elementwise. 
The continuity of $G^\rho$ implies that $G^{\rho,m}$ converges to $G^\rho$ uniformly on compact sets as $m\to \infty$. Moreover, one can easily  deduce from \eqref{eq:mono_p} and the properties of mollifiers $(\delta_m)_{m\in \N}$  that both $G^\rho$  and $(G^{\rho,m})_{m\in \N}$ satisfy the monotonicity condition \eqref{eq:mono} with the same constant $\gamma$.

For any given $m\in \N$, we shall now  apply Lemma \ref{lemma:homeo} to establish that $G^{\rho,m}$ is a homeomorphism from $\R^{Nd}$ onto $\R^{Nd}$, which implies the equation $G^{\rho,m}=0$ has a solution. More precisely, we shall show (1) the Jacobian matrix $\nabla G^{\rho,m}(u)$ is nonsingular for any given $u\in \R^{Nd}$ and (2) $\lim_{\|u\|\to \infty}\|G^{\rho,m}(u)\|=\infty$. To prove (1), suppose $\nabla G^{\rho,m}(u)x=0$ for some $u,x\in \R^{Nd}$ and let $|x^i_l|=\|x\|\ge 0$ for some $i\in \cI$ and $l\in\cN$.  If $x^i_l \ge 0$, we can deduce from the differentiability and monotonicity of $G^{\rho,m}$ that 
$$
\gamma(hx^i_l-0)\le G_i^{\rho,m}(u+hx)_l-G^{\rho,m}_i(u)_l-h(\nabla G_i^{\rho,m}(u)x)_l=\cO(h^2),\q \textnormal{as $h\to 0$},
$$
which implies $\|x\|=0$. The same conclusion can be drawn for the case with $x^i_l \le 0$, which implies $x=0$ and consequently the non-singularity of $\nabla G^{\rho,m}(u)$. To prove (2), let $u\in \R^{Nd}$ and $|u^i_l|=\|u\|$ for some $i\in \cI$ and $l\in\cN$. If $u^i_l \ge 0$, we can obtain from the monotonicity of $G^{\rho,m}$ that 
$$
\|G^{\rho,m}(u)\|\ge G_i^{\rho,m}(u)_l\ge \gamma(u^i_l)+G_i^{\rho,m}(0)_l\ge  \gamma \|u\|-\|G^{\rho,m}(0)\|,
$$
where the same estimate can be derived similarly for the case $u^i_l \le 0$. Therefore, we can conclude the existence of a solution $u^{\rho,m}$ to $G^{\rho,m}=0$. Since $G^{\rho,m}$ satisfies  \eqref{eq:mono} with the same constant $\gamma$, one can deduce from Lemma \ref{lem:bound_p} and the continuity of $G^{\rho,m}$ that its solution is  uniformly bounded, i.e., $\|u^{\rho,m}\|\le\|G^{\rho,m}(0)\|/\gamma\le L$ independent of $m\in \N$. 

Lastly, let $(u^{\rho,m_k})_{k\in \N}$ be a convergent subsequence of $(u^{\rho,m})_{m\in \N}$ with a limit $u^\rho$. 
Note that 
\begin{align*}
|G^\rho(u^\rho)-G^{\rho,m_k}(u^{\rho,m_k})|\le |G^\rho(u^\rho)-G^{\rho}(u^{\rho,m_k})|
+|G^{\rho}(u^{\rho,m_k})-G^{\rho,m_k}(u^{\rho,m_k})|\to 0,
\end{align*}
as $k\to \infty$, due to the continuity of $G^\rho$ and the uniform convergence  (on compact sets) of $G^{\rho,m}$  to $G^\rho$.
Therefore, $u^\rho$ is a solution of the penalized equation \eqref{eq:penalty} $G^\rho=0$.
%
\end{proof}

We end this section  with the following monotone convergence result of the  penalty approximations.
\begin{Theorem}\l{thm:mono_conv}
For any fixed switching cost $c\ge 0$, the solution to Problem \ref{pro:penalty}  converges monotonically from below to 
a function $u\in \R^{N\t d}$ as the penalty parameter $\rho\to \infty$. Moreover, $u$ solves Problem \ref{pro:switching} if the  switching cost $c$ is positive.
\end{Theorem}
\begin{proof}
It is straightforward to verify that if $u^{\rho_1}$ satisfies \eqref{eq:penalty} with the parameter $\rho_1$ and $\rho_1\ge \rho_2\ge 0$, then $G^{\rho_2}(u^{\rho_1})\ge 0$. Hence one can deduce from Proposition \ref{prop:comparison_p} that $u^{\rho_1}\ge u^{\rho_2}$, which together with Lemma \ref{lem:bound_p} implies $u^{\rho}$ converges to some function $u\in \R^{N\t d}$ as $\rho\to \infty$. Owing to the fact that the solution of \eqref{eq:switching} is unique for positive switching cost, it suffices to show $u$ solves Problem \ref{pro:switching}.

Let $i\in \cI$ be fixed. Since $\|u^{\rho}\|\le \|F(0)\|/\gamma$, we see  that 
$\sum_{j\not =i}\pi( u^{ \rho,j}_l-c-u^{ \rho,i}_l)\le C/\rho$ for all  $l\in \cN$, with a constant $C$ defined as:
\bb\l{eq:C}
C=\sup_{\|u\|\le  \|F(0)\|/\gamma} \|F(u)\|<\infty,
\ee
which is finite due to the continuity of $F$. 
Hence the limiting function $u$ satisfies $ u^{ i}\ge \max_{j\not=i}(u^{ j}-c)$ and $F_i(u)=\lim_{\rho\to\infty}F_i(u^\rho)\ge 0$.
Moreover, suppose $u^{i}_l -(\cM_iu)_l>0$ at the index $l\in \cN$, we can deduce that
$F_i(u^\rho)_l=0$ for all large enough $\rho$,
which further implies $F_i(u)_l= 0$ and completes our proof.
\end{proof}
\begin{Remark}\l{rmk:wp}
We point out that, unlike  Problem  \ref{pro:switching}, the well-posedness of Problem \ref{pro:penalty},  and the monotone convergence of their solutions  hold for any non-negative switching cost $c$, which enables us to study penalty schemes with zero switching cost (see Section \ref{sec:error_zero}).

Moreover, based on the penalized equations, Theorems \ref{thm:wp_p} and \ref{thm:mono_conv} explicitly construct the  solution to Problem \ref{pro:switching}, which is uniformly bounded by $\|F(0)\|/\gamma$ for all  positive switching costs. 
\end{Remark}

\section{Penalization errors for  positive switching cost}\l{sec:error_positive}

In this section, we shall proceed to analyze the convergence rate of the penalty approximation for Problem \ref{pro:switching} with a fixed positive switching cost. As discussed in Section  \ref{sec:introduction}, it is not easy to construct a supersolution of Problem \ref{pro:switching} from the solution of Problem \ref{pro:penalty} due to the non-diagonal dominance of the interconnected obstacles. We shall overcome this difficulty by regularizing the obstacles and establish the  convergence  rates of the penalty approximations with respect to the penalty parameter.

In order to obtain error estimates of the regularization procedures, we  impose the following concavity condition on the monotone system:
\begin{Assumption}\l{assum:concave}
The function $F$ in Problem \ref{pro:switching} is concave in the sense that: for any given $i\in\cI$, $u,v\in \R^{N\t d}$, $\theta\in [0,1]$, we have $F_i(\th u+(1-\th)v)\ge \th F_i(u)+(1-\th)F_i(v)$.
\end{Assumption}

Assumption \ref{assum:concave}  will only be used in Section \ref{sec:error_positive} to quantify the regularization errors (not for the well-posedness or the monotone convergence of the regularization procedures).
It is well-known that a concave function can be equivalently represented as the infimum of a family of affine functions, i.e., $F_i(u)=\inf_{\a\in \cA_i}B_i(\a)u-b_i(\a)$ for some set $\cA_i$ and coefficients $B_i:\cA_i\to \R^{Nd\t Nd}$ and $b_i:\cA_i\to \R^{Nd}$, hence our error estimates apply to the HJBQVIs studied in \cite{boulbrachene2001fem,seydel2009, azimzadeh2016weakly,ferretti2017,azimazadeh2018}. However, our setting significantly extends the classical HJBQVIs in the following important aspects: (1)  $F_i$ can depend on all  components of the solutions to the switching systems, (2) the control set $\cA_i$ can be non-compact  and coefficients $B_i$, $b_i$ can be discontinuous, (3) $b_i$ does not necessarily have a unique sign.

\subsection{Regularizations of the QVIs}\l{sec:regularization}
In this section we discuss how to approximate Problem \ref{pro:switching} by variational inequalities with diagonally dominant obstacle terms. We shall propose two regularization procedures, namely an iterated optimal stopping approximation and a novel time-marching iteration, and estimate the regularization errors, which are essential for analyzing the  penalization error of Problem \ref{pro:penalty}.

Similar error estimates of  the iterated optimal stopping approximation have been obtained in \cite{bonnans2007,ferretti2017} for continuous (scalar-valued) elliptic HJBQVIs with positive running costs, finite control sets, and sufficiently regular coefficients.  Here, we  relax these conditions and obtain regularization errors for general  discrete monotone systems satisfying Assumption \ref{assum:concave}. The time-marching regularization leads to a more accurate approximation to Problem \ref{pro:switching} than the iterated optimal stopping regularization, especially when the switching cost is small.

Let us start with the iterated optimal stopping approximation (see \cite{seydel2009,ferretti2017} for its applications to the classical HJBQVIs), 
which  approximates Problem \ref{pro:switching} as follows:
find  $u^{0}\in \R^{N\t d}$ satisfying
$F_i(u^{0})=0$, $i\in \cI$, and for each $n\ge 1$, given $u^{n-1}\in \R^{N\t d}$, find $u^{n}\in \R^{N\t d}$ such that $u^{n}=Qu^{n-1}$, where 
  for any given $u\in \R^{N\t d}$, we define $Qu\coloneqq ((Qu)^1,\ldots, (Qu)^d)\in \R^{N\t d}$ to be the quantity which satisfies the following obstacle problem:
\bb\l{eq:switching_itr}
\min(F_i(Qu),\; (Qu)^{i}-\cM_i u)=0,\q i\in \cI.
\ee

%

By extending the arguments in Section \ref{sec:penalty}, one can  show that the above procedure is well-defined. Moreover, it is not difficult to  establish   the following comparison principle for \eqref{eq:switching_itr}: for any fixed $w\in  \R^{N\t d}$, if $u\in  \R^{N\t d}$ satisfies 
$$
\min(F_i(u),\; u^{i}-\cM_i w)\le 0,\q i\in \cI,
$$
and $v\in  \R^{N\t d}$ satisfies 
$$
\min(F_i(v),\; v^{i}-\cM_i w)\ge 0,\q i\in \cI,
$$
then $u\le v$. 
In fact, for any fixed $w\in  \R^{N\t d}$, we  can show that  the system $F_w=(F_{w,i})_{i\in\cI}$ such that $F_{w,i}(u)\coloneqq \min(F_i(u),\; u^{i}-\cM_i w)$, $i\in \cI$, satisfies the monotone condition \eqref{eq:mono} with the constant $\min(\gamma,1)$, which subsequently implies the above comparison principle.
\color{black}

The next result presents some important properties of the  operator $Q$.

\begin{Lemma}\l{lem:Q}
The operator $Q$ is monotone, i.e., $Qu\ge Qv$ provided that $u\ge v$, and satisfies the a priori estimate:
$$
\|Qu\|\le \max(\|F(0)\|/\gamma,\, \|u\|) \q  \fa u\in \R^{N\t d}.
$$
If we further suppose Assumption \ref{assum:concave} holds, then the operator $Q$ is convex.
\end{Lemma}
\begin{proof}
If $u\ge v$, then $-\cM_i u\le -\cM_i v$ due to the monotonicity of $\cM_i$. Thus, we have that
$$
 \min(F_i(Qu),\; (Qu)^{i}-\cM_i v)\ge \min(F_i(Qu),\; (Qu)^{i}-\cM_i u)=0 ,\q i\in \cI,
$$
which together with the comparison principle of \eqref{eq:switching_itr} shows that $Qu\ge Qv$.

For the a priori estimate, we suppose that $|(Qu)^i_l|=\|Qu\|$. 
If $(Qu)^i_l=(\cM_i u)_l$ and $ (\cM_i u)_l\ge 0$, then we have $\|Qu\|=(\cM_i u)_l\le \|u\|$. 
Otherwise, we can adapt the arguments of Lemma \ref{lem:bound_p} to show $\|Qu\|\le \|F(0)\|/\gamma$. Finally, for any given $u,v\in \R^{N\t d}$ and $\theta\in [0,1]$, one can deduce from the concavity of $F_i$ and $-\cM_i$ that $\th Qu+(1-\th)Qv$ is a supersolution to \eqref{eq:switching_itr} with the obstacle $\cM_i(\th u+(1-\th)v)$, hence the 
comparison principle  leads us to the desired result.
\end{proof}

The above lemma directly implies the monotone convergence of the iterates $(u^n)_{n\in \N}$.
\begin{Proposition}\l{prop:iterated_conv}
For any given positive switching cost $c> 0$, the iterates $(u^n)_{n\in \N}$ satisfy $\|u^n\|\le \|F(0)\|/\gamma$ for all $n\in \N$, and converge monotonically from below to 
the solution $u$ of Problem \ref{pro:switching} as $n\to \infty$. 
\end{Proposition}
\begin{proof}
The bound of $u^0$ follows from Lemma \ref{lem:bound_p} (with $\rho=0$), while  
the uniform bound of $(u^n)_{n\in \N}$ follows from Lemma \ref{lem:Q}.
Moreover, since $F(u^0)=0$ and $F(u^1)\ge 0$ implies $u^1\ge u^0$ by the comparison principle of $F$,
 we can show by an inductive argument and the monotonicity of the operator $Q$ that $(u^n)_{n\in \N}$ monotonically increases to some vector $u$, which solves Problem \ref{pro:switching} due to the continuity of \eqref{eq:switching_itr}.
\end{proof}

Now we proceed to estimate the difference $u-u^{n}$, where $u$ and $u^{n}$ solve Problem \ref{pro:switching} and the equation \eqref{eq:switching_itr}, respectively. We shall first introduce the concept of strict supersolution, which was used in
\cite{ishii1993,seydel2009} to study impulse control problems.
\begin{Definition}\l{def:switching_strict}
 A vector $w\in \R^{N\t d}$ is said to be a strict supersolution of Problem \ref{pro:switching} if there exists  a  constant $\kappa>0$, such that $G_i(w)=\kappa$ for all $i\in \cI$.
\end{Definition}

For any any given $0<\kappa<c$, by applying Theorem \ref{thm:wp_p} to the problem
$$
\min (F_i(u)-\kappa,\; u^i-\cM^\kappa_{i} u)=0, \q \textnormal{with $\cM^\kappa_i u\coloneqq \max_{j\not = i}(u^j-(c-\kappa))$,}
$$
we can show that Problem \ref{pro:switching} admits a unique strict supersolution satisfying the bound $\|w\|\le(\|F(0)\|+\kappa)/\gamma$. For convenience,  we shall assume without loss of generality that  $\|F(0)\|>0$ in the remaining part of this paper, which excludes the trivial case where $0$ is  the unique solution to Problems \ref{pro:switching} and \ref{pro:penalty}.

The next lemma shows a contractive property of the  operator $Q$. 
A similar result has been shown in \cite{seydel2009}  for a classical (continuous in time and space) HJBQVI via a control-theoretic approach. 
Here we shall present a simpler proof for our discrete setting based on the comparison principle, which can be easily extended to other regularization methods. For any given $\kappa\in (0,c)$, we  introduce the following constant $L_\kappa$, which will be used frequently in the subsequent analysis:
\bb\l{eq:L}
L_\kappa\coloneqq (2\|F(0)\|+\kappa)/\gamma.
\ee

\begin{Lemma}\l{lem:contraction}
Suppose Assumption \ref{assum:concave} holds. Let  $w$ be the strict supersolution to Problem \ref{pro:switching} with $\kappa\in (0,c)$. If $u^{n}-u^{n-1}\le \lambda (w-u^{n-1})$ for some $\lambda\in [0,1]$ and $n\in \N$, then
 we have 
$u^{n+1}-u^{n}\le \lambda (1-\mu) (w-u^{n})$ with
\bb\l{eq:mu}
\mu=\min\bigg (1,\f{\gamma \kappa}{2\|F(0)\|+\kappa}\bigg)\in (0,1].
\ee
\end{Lemma}
\begin{proof}
One can deduce from  the  convexity  and monotonicity of the  operator $Q$ that
$$
Qu^{n}\le Q( \lambda w+(1-\lambda) u^{n-1})\le  \lambda Q  w+(1-\lambda) Qu^{n-1}=\lambda Q  w+(1-\lambda) u^{n},
$$
hence it suffices to show $Q  w\le (1-\mu)w+\mu u^{n}$.
Note that for any given $i\in \cI$, we can obtain from the concavity of $F_i$ that $F_i((1-\mu)w+\mu u^{n})\ge 0$, and also for $\mu\in [0,1]$,
\begin{align*}
(1-\mu)w^{i}+\mu u^{n,i}-\cM_i w\ge \kappa-\mu (w^{i}+(u^{n,i})^-)\ge  \kappa-\mu (\|w\|+\|(u^{n})^-\|)\ge 0,
\end{align*}
provided that $\mu\le \f{\kappa}{\|w\|+\|(u^{n})^-\|}$. Since  we have 
$\|w\|+ \|(u^{n})^-\|\le  L_\kappa$ for all $n\in \N$, 
  setting $\mu= \min(1,{\kappa}/{L_\kappa})$ gives us that
$$
\min (F_i((1-\mu)w+\mu u^{n}),\; (1-\mu)w^{i}+\mu u^{n,i}-\cM_i w)\ge 0, \q i\in \cI,
$$  
which  subsequently enables us to conclude
the desired result   from the comparison principle.
\end{proof}

Now we are ready to present the error estimate of the iterated optimal stopping approximation.

\begin{Theorem}\l{thm:iterated_rate}
Suppose Assumption \ref{assum:concave} holds, 
$u$ solves  Problem \ref{pro:switching}, and  $(u^{n})_{n\in\N}$ are recursively defined  by 
the equation \eqref{eq:switching_itr}. 
Then we have for any given $\kappa\in (0,c)$ that
$$
0\le u-u^{n}\le L_\kappa(1-\mu)^n/\mu, \q \fa n\ge 0,\!\footnotemark
$$
where $\mu$ and $L_\kappa$ are defined as in \eqref{eq:mu} and \eqref{eq:L} respectively.
\footnotetext{Here we follow the convention that $0^0=1$.}
\end{Theorem}
\begin{proof}
Let  $w$ be the strict supersolution to Problem \ref{pro:switching} with parameter $\kappa$. Since $(u^n)_{n\in\N}$ converge to $u$ monotonically from below, the comparison principle for Problem \ref{pro:switching} applied to $u$ and $w$ implies that $u^0\le u^1\le \cdots\le u\le w$.
Hence $u^1-u^0\le \lambda(w-u^0)$ with $\lambda=1$, which along with Lemma \ref{lem:contraction} gives that $u^2-u^1\le \lambda(1-\mu)(w-u^1)$. Inductively, we have
$$
0 \le u^{n+1} -u^{n}\le (1-\mu)^n (w- u^{n})\le (1 -\mu)^n(w-u^{0}), \q n\ge 0.
$$
Now summing the above inequality and employing $w-u^0\le L_\kappa$, we obtain that
\begin{equation*}
0\le u-u^n=\sum_{k=n}^\infty u^{k+1}-u^k\le (w-u^{0})\sum_{k=n}^\infty (1 -\mu)^k\le L_\kappa (1-\mu)^n/\mu,  \q n\ge 0,
\end{equation*}
which gives us the desired estimate and completes the proof.
\end{proof}
\begin{Remark}
Similar geometric convergence rates have been establish in \cite{boulbrachene2001fem,ferretti2017} for classical HJBQVIs, i.e., $F_i(u)=\inf_{\a\in\cA_i} B_i(\a)u^i-b_i(\a)$ for all $i\in\cI$, under the assumptions that $\cA_i$ is a compact (or finite) set and $b_i(\a)\ge 0$ for all $\a\in \cA_i$. Here we remove these restrictions.

Theorem \ref{thm:iterated_rate} suggests the iterated optimal stopping  only gives a good approximation to Problem \ref{pro:switching} for sufficiently large  switching cost $c$. For small enough switching cost $c$, we have $1-\mu\approx 1-\gamma c/(2\|F(0)\|)$, which converges to $1$ as $c\to 0$.

\end{Remark}

Due to the slow convergence rate of the iterated optimal stopping approximation for small switching cost, let us now discuss another regularization method, called the time-marching iteration (see \cite{ito2001}), which introduces an additional pseudo-time parameter $\eps$ to the interconnected obstacle, and gives an accurate approximation to Problem \ref{pro:switching} even for small switching cost.

For any given parameter $\eps>0$, the time-marching iteration is given as follows:
find  $u^{0}\in \R^{N\t d}$ satisfying
$F_i(u^{0})=0$, $i\in \cI$, and for each $n\ge 1$, given $u^{n-1}\in \R^{N\t d}$, find $u^{n}\in \R^{N\t d}$ such that $u^{n}=Tu^{n-1}$, where 
 for any given $u\in \R^{N\t d}$, we define $Tu\coloneqq ((Tu)^1,\ldots, (Tu)^d)\in \R^{N\t d}$ to be the quantity which satisfies the following obstacle problem:
\bb\l{eq:pseudo_itr}
\min\big(F_i(Tu),\; (Tu)^{i}-\cM_i (Tu)+\eps((Tu)^i-u^i)\big)=0,\q i\in \cI.
\ee
%

The  operator $T$ enjoys analogue properties as the operator $Q$, i.e., Lemma \ref{lem:Q} and Proposition \ref{prop:iterated_conv}, whose proofs are similar and details are omitted. Moreover, similar to \eqref{eq:switching_itr}, we can  show \eqref{eq:pseudo_itr} admits the following comparison principle: for any fixed $w\in  \R^{N\t d}$, if $u\in  \R^{N\t d}$ satisfies 
$$
\min\big(F_i(u),\; u^{i}-\cM_i u+\eps(u^i-w^i)\big)\le 0,\q i\in \cI.
$$
and $v\in  \R^{N\t d}$ satisfies 
$$
\min\big(F_i(v),\; v^{i}-\cM_i v+\eps(v^i-w^i)\big)\ge 0,\q i\in \cI.
$$
then $u\le v$. 

The next theorem presents the convergence rate of the time-marching iteration.

\begin{Theorem}\l{thm:pseudo_rate}
Suppose Assumption \ref{assum:concave} holds, 
$u$ solves  Problem \ref{pro:switching}, and  $(u^{n})_{n\in\N}$ are recursively defined  by the equation \eqref{eq:pseudo_itr} with the pseudo-time parameter $\eps>0$. Then we have for any  $\kappa\in (0,c)$ that
$$
0\le u-u^{n}\le L_\kappa(1-\mu)^n/\mu, \q  n\ge 0,
$$
where  $L_\kappa$ is defined as in  \eqref{eq:L} and $\mu={\kappa}/{(\kappa+\eps L_\kappa)}$.
\end{Theorem}
\begin{proof}
Let  $w$ be the strict supersolution to Problem \ref{pro:switching} with parameter $\kappa$. 
Following the proofs of Lemma \ref{lem:contraction} and Theorem \ref{thm:iterated_rate}, we see it is essential to obtain $\mu\in(0,1]$ such that 
$T  w\le (1-\mu)w+\mu u^{n}$ for all $n\in \N$, which leaves us to show that for suitable $u$ we have
\begin{align*}
&(1-\mu)w^i+\mu u^{n,i}-\cM_i[(1-\mu)w+\mu u^{n}]+\eps((1-\mu)w^i+\mu u^{n,i}-w^i)\\
\ge \,&(1-\mu)(w^i-\cM_i w)+\mu [u^{n,i}-\cM_i u^{n}+\eps(u^{n,i}-u^{n-1,i})]-\eps \mu (w^i-u^{n-1,i})\\
\ge \,&(1-\mu)\kappa-\eps \mu (\|w\|+\|(u^{n-1})^-\|)\ge (1-\mu)\kappa-\eps \mu L_\kappa\ge 0,
\end{align*}
for $L_\kappa$ defined as in \eqref{eq:L}. Thus by setting $\mu={\kappa}/{(\kappa+\eps L_\kappa)}$, we see the above inequality holds, and one can deduce the desired result following similar arguments as the proof of Theorem \ref{thm:iterated_rate}.
\end{proof}

\begin{Remark}\l{rmk:pseudo}
Through the choice of the pseudo-time parameter $\eps$,  the time-marching iteration  gives a more accurate approximation to Problem \ref{pro:switching}  than the iterated optimal stopping approximation for small switching cost $c$. In fact, it holds for the time-marching iteration that
$$
1-\mu=1-\f{\kappa}{\kappa+\eps  (2\|F(0)\|+\kappa)/\gamma}=1-\f{1}{1+  2\eps\|F(0)\|/(\kappa\gamma)+\eps/\gamma}.
$$
Therefore, taking $c\to 0$ and $\eps\to 0$ such that $\eps/c\to 0$, we get $1-\mu\to 0$.
However, as we shall see in Section \ref{subsec:error_positive}, the error of the penalty approximation to \eqref{eq:pseudo_itr} grows proportionally to $1/\eps$, hence after minimizing over $\eps$, both the iterated optimal stopping approximation and the time-marching iteration lead to the same error estimate for Problem \ref{pro:penalty}.
\end{Remark}

\subsection{Convergence order of penalty methods}\l{subsec:error_positive}

In this section, we shall use the regularization procedures proposed in Section \ref{sec:regularization} to demonstrate that for fixed positive switching cost and a penalty function with degree $\sigma>0$, the approximation error of Problem \ref{pro:penalty} is bounded above by the quantity $C_0\rho^{-\sigma}\ln\rho$ for some constant $C_0$, which depends only on the function $F$ and is independent of the number of switching regimes $d$. Since both regularization procedures lead to the same error estimate (see Remark \ref{rmk:pseudo}), we shall focus on the regularization by the iterated optimal stopping, and only outline the essential results for  the time-marching iteration. 

To quantify the penalty error of Problem \ref{pro:penalty} with a fixed  parameter $\rho \ge 0$, we  introduce the following sequence of auxiliary problems:
find  $u^{\rho,0}\in \R^{N\t d}$ satisfying
$F_i(u^{0})=0$, $i\in \cI$, and for each $n\ge 1$, given $u^{\rho,n-1}\in \R^{N\t d}$, find $u^{\rho, n}\in \R^{N\t d}$ such that $u^{\rho,n}=Q^\rho u^{\rho,n-1}$, where 
 for any given $u\in \R^{N\t d}$, we define $Q^\rho u\coloneqq ((Q^\rho u)^1,\ldots, (Q^\rho u)^d)\in \R^{N\t d}$ to be the quantity which satisfies the following penalized equation:
\bb\l{eq:penalty_itr_op}
F_i(Q^\rho u)-\rho\sum_{j\not =i}\pi( u^j-c-(Q^\rho u)^i )=0,\q i\in \cI.
\ee


Note that the above auxiliary problem has the same initialization as the iterated optimal stopping approximation to Problem \ref{pro:switching}, i.e., $u^{\rho,0}=u^0$. 
Moreover, for any fixed $w\in  \R^{N\t d}$, we can consider the system $G_w=(G_{w,i})_{i\in\cI}$ such that  $G_{w,i}(u)\coloneqq F_i(u)-\rho\sum_{j\not =i}\pi( w^j-c-u^i )$, $i\in \cI$, which satisfies the monotone condition \eqref{eq:mono} with the constant $\gamma$ due to the facts that $F$ is monotone and $\pi$ is non-decreasing. Consequently, we have the following comparison principle for 
\eqref{eq:penalty_itr_op}: for any fixed $w\in  \R^{N\t d}$, if $u\in  \R^{N\t d}$ satisfies $G_{w,i}(u)\le 0$, $i\in \cI$, and 
$v\in  \R^{N\t d}$ satisfies $G_{w,i}(v)\ge 0$, $i\in \cI$, then $u\le v$. Therefore, we can easily establish  the well-posedness of \eqref{eq:penalty_itr_op} by adapting the proof of Theorem \ref{thm:wp_p}. 

The following result summarizes the essential properties of the operator $Q^\rho$ and the iterates $(u^{\rho,n})_{n\in\N}$.

%
\begin{Proposition}\l{lem:Q^rho}
The operator $Q^\rho$ is monotone, satisfies the a priori estimate:
$$
\|Q^\rho u\|\le \max(\|F(0)\|/\gamma,\, \|u\|) \q  \fa u\in \R^{N\t d},
$$
and is Lipschitz continuous with constant 1, i.e., $\|Q^\rho u-Q^\rho v\|\le \|u-v\|$ for all $u,v\in \R^{N\t d}$.
Consequently, for any given penalty parameter $\rho \ge 0$ and switching cost $c\ge 0$, the iterates $(u^{\rho,n})_{n\in \N}$ 
converge monotonically from below to the solution $u^\rho$ of Problem \ref{pro:penalty} as $n\to \infty$. 
\end{Proposition}
\begin{proof}
The a priori bound can be obtain exactly as Lemma \ref{lem:Q}.
For the monotonicity and Lipschitz continuity of $Q^\rho$, it suffices to show  for any given $u,v\in \R^{N\t d}$, we have
$Q^\rho u-Q^\rho v\le \| (u-v)^+\|$.

For any given $u,v\in \R^{N\t d}$, we introduce the quantity $\hat{u}\coloneqq Q^\rho v+ \| (u-v)^+\|$.
It is important to observe that for any given  $L\ge 0$ and $u\in \R^{N\t d}$, the monotonicity \eqref{eq:mono} of $F$ implies 
\bb\l{eq:translation}
F_i(u+L)-F_i(u)\ge \gamma L \q \fa i\in \cI,
\ee
which along with the fact that
$$
\pi( u^{j}-c-\hat{u}^{i})=\pi( u^{j}- \| (u-v)^+\|-c-(Q^\rho v)^i)\le \pi( v^{j}-c-(Q^\rho v)^i), \q j\not =i,
$$
enables us to conclude the desired result through the following estimate: for any $i\in \cI$,
\begin{align*}
F_i(\hat{u})-\rho\sum_{j\not =i}\pi( u^j-c-\hat{u}^i)\ge F_i(Q^\rho v)+\gamma \| (u-v)^+\|-\rho\sum_{j\not =i}\pi( v^{j}-c-(Q^\rho v)^i)\ge 0.
\end{align*}
The above implies that $\hat{u}$ is a supersolution to \eqref{eq:penalty_itr_op} with the input $u$. By the comparison principle for \eqref{eq:penalty_itr_op}, $Q^\rho u\le \hat{u}$ and hence $Q^\rho u-Q^\rho v\le \| (u-v)^+\|$ as desired.
Then the monotone convergence of $(u^{\rho,n})_{n\in \N}$ follows from similar arguments as those in Proposition \ref{prop:iterated_conv}.
\end{proof}

The next result provides an upper bound of the term $u^n-u^{\rho,n}$, where $u^n$ and $u^{\rho,n}$ solve  the equations \eqref{eq:switching_itr} and \eqref{eq:penalty_itr_op}, respectively.
\begin{Proposition}\l{prop:err_n}
For any given penalty parameter $\rho\ge 0$ and switching cost $c> 0$, let 
$(u^n)_{n\in\N}$ and $(u^{\rho,n})_{n\in\N}$ be recursively defined by  the equations \eqref{eq:switching_itr} and \eqref{eq:penalty_itr_op}, respectively. 
Suppose that there exist positive constants $\tau$ and $\sigma$ such that $\pi(y)\ge \tau y^{1/\sigma}$ for all $0\le y\le 2\|F(0)\|/\gamma$. Then we have 
$$
\|u^n-u^{\rho,n}\|\le \left(\f{C}{\tau\rho}\right)^\sigma n, \q n\ge 0,
$$
where $C=\sup_{\|u\|\le  \|F(0)\|/\gamma} \|F(u)\|$.
\end{Proposition}
\begin{proof}
The Lipschitz continuity of $Q^\rho$ implies that for any $n\in \N$,
\bb\l{eq:induction}
\|u^n-u^{\rho,n}\|=\|u^{n}-Q^\rho u^{n-1}\|+\|Q^\rho u^{n-1}-Q^\rho u^{\rho, n-1}\|\le \|u^{n}-Q^\rho u^{n-1}\|+\| u^{n-1}- u^{\rho, n-1}\|.
\ee

Now we bound $u^{n}-Q^\rho u^{n-1}$ for any given $n\in \N$. 
From the  a priori bounds of $u^n$ (Proposition \ref{prop:iterated_conv}) and $Q^\rho$ (Proposition \ref{lem:Q^rho}), we know   $\|Q^\rho u^{n-1}\|\le \|F(0)\|/\gamma$ for all $\rho\ge 0$ and  $n\in \N$.
Moreover, by using the comparison principle for \eqref{eq:penalty_itr_op} and a modification of the arguments in Theorem \ref{thm:mono_conv}, we can deduce for any given $n\in \N$ that  $(Q^\rho u^{n-1})_{\rho>0}$ converges monotonically from below to $u^n$ as  $\rho\to \infty$. 
This implies that $\|\rho\pi( u^{n-1,j}-c-(Q^\rho u^{n-1})^i)\|\le C$ for all $j\not =i$, $i\in \cI$, where $C$ is defined as in \eqref{eq:C}. Therefore, we have 
\begin{align*}
(Q^\rho u^{n-1})^i+\bigg(\f{C}{\tau\rho}\bigg)^\sigma&-\cM_i u^{n-1}=\min_{j\not =i}\f{1}{\rho^\sigma}\left(\bigg(\f{C}{\tau}\bigg)^\sigma-\rho^\sigma(u^{n-1,j}-c-(Q^\rho u^{n-1})^i) \right)\\
&\q \ge\f{1}{(\rho\tau)^\sigma} \min_{j\not =i}\left(C^\sigma-\|\rho\pi(u^{n-1,j}-c-(Q^\rho u^{n-1})^i)\|^\sigma\right)\ge 0.
\end{align*}
Moreover, by applying \eqref{eq:translation} with $u=Q^\rho u^{n-1}$ and $L=({C}/{(\tau\rho)})^\sigma$, we can obtain that 
$$
F_i\bigg(Q^\rho u^{n-1}+\bigg(\f{C}{\tau\rho}\bigg)^\sigma\bigg)\ge F_i(Q^\rho u^{n-1})+ \gamma \bigg(\f{C}{\tau\rho}\bigg)^\sigma\ge 0, \q \fa i\in \cI,
$$
which implies $Q^\rho u^{n-1}+\big({C}/{(\tau\rho)}\big)^\sigma$ is a supersolution to \eqref{eq:switching_itr}. Consequently, we obtain  
$0\le u^{n}-Q^\rho u^{n-1}\le ({C}/{(\tau\rho)})^\sigma$ for all $n\in \N$, and conclude the desired result from $u^{\rho,0}=u^0$ and \eqref{eq:induction}.
\end{proof}

\begin{Remark}\l{rmk:penalty}
This proposition  greatly extends the results in \cite{witte2012} (even for the case with $\sigma=1$) by removing the continuous differentiability assumption of the penalty function $\pi$.  In practice, one can choose $\pi(y)= (y^+)^{1/\sigma}$ as the penalty function. Since $\pi$ is  semismooth  if $\sigma=1$,  
a direct application of semismooth Newton methods allows us to solve Problem \ref{pro:penalty} efficiently (see \cite{witte2012, reisinger2018}).  A penalty term with $\sigma>1$ needs an additional smoothing for the application of Newton methods and usually requires a larger number of Newton iterations to solve the penalized equation \cite{huang2010}.
Though the higher  convergence rate allows us to use a relatively small value of $\rho$ to achieve the desired accuracy, which could avoid the numerical instability  caused by the usage of a large penalty parameter according to \cite{huang2010}, we do not discover any problem using the penalty function $\pi(y)= y^+$ in our numerical experiments.
\end{Remark}

Now we are ready to conclude the penalty error of Problem \ref{pro:penalty} to Problem \ref{pro:switching}. The following result has been proved in \cite[Lemma 6.1]{bonnans2007}, and will be used in our error estimates.
\begin{Lemma}\l{lemma:phi}
 Let $\phi:\R\to \R$, $\phi(x)=\nu a^x+bx$, where $0<a<1$, $0<b<\infty$ and $\nu > 0$. Let $m\coloneqq \min_{n\in \N}\phi(n)$. Then we have
 $$
m\le \begin{cases} \nu, & -b/(\nu \ln a)\ge 1,\\
-ab/(\ln a)+b[\log_a (b/(\nu \ln a))+1], &\textnormal{otherwise}.
\end{cases}
$$
\end{Lemma}

\begin{Theorem}\l{thm:panalty_rate}
For any given switching cost $c> 0$, let $u$ and $u^{\rho}$ solve  Problem \ref{pro:switching} and \ref{pro:penalty}, respectively. 
Suppose that Assumption \ref{assum:concave} and the assumptions in  Proposition \ref{prop:err_n}  hold. Then if $c>2\|F(0)\|/\gamma$, we have $u^\rho=u$ for all $\rho\ge 0$, and if $c\le 2\|F(0)\|/\gamma$, we have for any $\kappa\in (0,c)$, 
\begin{equation*}
\|u-u^{\rho}\|\le f(\rho), \q \textnormal{where}\q f\sim -\f{\sigma (C/\tau)^\sigma}{\ln (1-\kappa/L_\kappa)}\rho^{-\sigma }\ln \rho,\q \textnormal{as $\rho\to \infty$,\footnotemark}
\end{equation*}
with the constants $C$  and $L_\kappa$  defined as in  \eqref{eq:C} and \eqref{eq:L} respectively.
\footnotetext{Recall that $f\sim g$ as $\rho\to \infty$ if $\lim_{\rho\to \infty} f(\rho)/g(\rho)=1$.}
\end{Theorem}
\begin{proof}
Suppose that $c>2\|F(0)\|/\gamma$. Theorem  \ref{thm:wp_p} shows that $\|u^\rho\|\le \|F(0)\|/\gamma$ for all $\rho\ge 0$, which  implies that
$$
u^{\rho,j}_l-c-u^{\rho,i}_l\le 2\|F(0)\|/\gamma-c< 0, \q  \fa i\in \cI, \, j\not =i,\, l\in\cN.
$$
Thus we have $G^\rho_i(u^\rho)= F_i(u^{\rho})=0$ for all $i\in \cI$ and $\rho\ge 0$. Similarly, we can obtain by using $\|u\|\le \|F(0)\|/\gamma$ (see Remark \ref{rmk:wp}) that $u^i-\cM_i u>0$, which implies that $F_i(u)=0$ for all $i\in\cI$. The comparison principle of $F$ gives us that $u=u^\rho$ for all $\rho\ge 0$. 

Now we assume that $c\le 2\|F(0)\|/\gamma$.
Since $u^\rho\le u$ for all $\rho\ge 0$, it remains to derive an upper bound of $u-u^\rho$. Note that 
$$
u-u^\rho\le u-u^n+u^n-u^{\rho,n}+u^{\rho,n}-u^\rho, 
$$
where  $u^n$ and $u^{\rho,n}$ are recursively defined by the equations \eqref{eq:switching_itr} and \eqref{eq:penalty_itr_op}, respectively. 
Proposition  \ref{lem:Q^rho} implies that $u^{\rho,n}\le u^\rho$ for all $\rho\ge 0$ and $n\in \N\cup\{0\}$.
Hence we can deduce from Theorem \ref{thm:iterated_rate} and Proposition \ref{prop:err_n}  that for any $\kappa\in (0,c)$,
\bb\l{eq:u-u^rho}
\|u-u^\rho\|\le L_\kappa\f{(1-\mu)^n}{\mu}+\left(\f{C}{\tau\rho}\right)^\sigma n  \q  \fa n\in\N\cup\{ 0\},
\ee
%
where $C$  and $L_\kappa$ are defined as in  \eqref{eq:C} and \eqref{eq:L} respectively, and $\mu=\min(1,\kappa/L_\kappa)<1$ due to the assumption that $c\le 2\|F(0)\|/\gamma$. 
Now we minimize the right-hand side of \eqref{eq:u-u^rho} over $n$ by applying Lemma \ref{lemma:phi} with $\nu=L_\kappa/\mu$, $a=1-\mu$ and $b=\big({C}/{(\tau\rho)}\big)^\sigma$.
If $\rho$ is sufficiently large, then $-b/(\nu \ln a)<1$, which implies that
$$
\|u-u^\rho\|\le 
-\f{1}{\ln (1-\mu)}\bigg(\f{C}{\tau\rho}\bigg)^\sigma\bigg[1-\mu- \ln(1-\mu)-\ln\bigg(\f{\mu}{L_\kappa \ln(1-\mu)}\bigg(\f{C}{\tau\rho}\bigg)^\sigma\bigg)\bigg].
$$
Thus by using the following identity:
$$
-\ln\bigg(\f{\mu}{L_\kappa \ln(1-\mu)}\bigg(\f{C}{\tau\rho}\bigg)^\sigma\bigg)=-\ln\bigg(\f{\mu}{L_\kappa \ln(1-\mu)}\bigg(\f{C}{\tau}\bigg)^\sigma\bigg)+\sigma \ln \rho,
$$
we deduce that $\|u-u^\rho\|\le f(\rho)$, where $f$ satisfies that 
$f\sim -\f{\sigma (C/\tau)^\sigma}{\ln (1-\mu)}\rho^{-\sigma }\ln \rho$, as $\rho\to \infty$.
\end{proof}

\begin{Remark}\l{rmk:ln}
Recall that $\mu=\cO(c)$ as $c\to 0$, hence the upper bound behaves as $\rho^{-\sigma}\ln \rho/c$ for small switching cost $c$. Unfortunately, we are not sure whether this  dependence on $c$ is optimal since the possible blow-up of the penalization error for small enough $c$ could be due to the fact that the iterated optimal stopping approximation does not provide an accurate approximation to Problem \ref{pro:switching} with small switching cost. Our numerical experiments  show that as the switching cost tends to zero,  the penalization error with a fixed penalty parameter indeed grows at a rate $\cO(c^{-1/2})$ for certain ranges of switching costs, but then stablizes to a limiting value 
(see Section \ref{sec:numerical}). As we shall see in Section \ref{sec:error_zero},  Problem \ref{pro:switching} degenerates into an HJB equation as the switching cost $c\to 0$, and  Problem \ref{pro:penalty} with $c=0$ provides a penalty approximation to such equation, with an asymptotical error $\cO(1/\rho^\sigma)$ as the penalty parameter $\rho\to \infty$.
Thus the penalization error with sufficiently small positive switching cost is dominated by this limiting error (see \eqref{eq:err_min}).
\end{Remark}

We proceed to  outline the key results of the convergence analysis by using the  time-marching iteration, and demonstrate that even if the time-marching iteration could improve the regularization error for small switching cost by adjusting the pseudo-time parameter $\eps$, it reduces the accuracy of penalty approximations at each iterate. Hence it leads to the same error estimate of Problem \ref{pro:penalty} as the iterated optimal stoping approximation for small switching cost.

For any given   parameters $\rho \ge 0$ and $\eps\ge 0$, we   introduce the following sequence of auxiliary problems:
find  $u^{\rho,0}\in \R^{N\t d}$ satisfying
$F_i(u^{0})=0$, $i\in \cI$, and for each $n\ge 1$, given $u^{\rho,n-1}\in \R^{N\t d}$, find $u^{\rho, n}\in \R^{N\t d}$ such that $u^{\rho,n}=T^\rho u^{\rho,n-1}$, where 
 for any given $u\in \R^{N\t d}$, we define $T^\rho u\coloneqq ((T^\rho u)^1,\ldots, (T^\rho u)^d)\in \R^{N\t d}$ to be the quantity which satisfies the following penalized equation:
$$
F_i(T^{\rho}u)-\rho\sum_{j\not =i}\pi\big( (T^{\rho}u)^{j}-c-(T^{\rho}u)^{i}-\eps((T^{\rho}u)^{i}-u^{i})\big)=0,\q i\in \cI.
$$
%
One can establish analogue results of Proposition \ref{lem:Q^rho} for the operator $T^\rho$, and demonstrate that 
$T^\rho u^{n-1}+(C_1/\rho)^\sigma/\eps$ is a supersolution to \eqref{eq:pseudo_itr} under the same assumptions of Proposition \ref{prop:err_n}. Therefore, following the proof of Theorem \ref{thm:panalty_rate}, we deduce from Theorem \ref{thm:pseudo_rate} that for any $\kappa\in (0,c)$, 
$$
\|u-u^\rho\|\le  L_\kappa\f{(1-\mu)^n}{\mu}+\f{1}{\eps}\left(\f{C}{\tau\rho}\right)^\sigma n  \q \fa n\in \N, \, \eps>0,
$$
where $C$  and $L_\kappa$ are defined as in  \eqref{eq:C} and \eqref{eq:L} respectively, and $\mu=\kappa/(\kappa+\eps L)$.
Minimizing over $n$, we obtain $\|u-u^\rho\|=\cO(-\f{\rho^{-\sigma}\ln\rho}{\eps\ln(1-\f{c}{c+\eps L_\kappa})})$ for all $\eps>0$ and large enough $\rho$. Therefore, by  discussing the cases $\eps=\cO(c)$ and $\eps/c\to \infty$  separately, we arrive at the same error estimate as that in Theorem \ref{thm:panalty_rate}.

Finally we  end this section with an exact construction of the optimal switching regions
\bb\l{eq:Gamma}
\Gamma_i\coloneqq \{l\in \cN \mid u^i_l=(\cM_iu)_l\}, \q i\in \cI,
\ee
 of Problem \ref{pro:switching} with  a  given switching cost $c>0$ using the solution of Problem \ref{pro:penalty}.
Suppose the estimate $0\le u-u^{\rho}\le C_0\rho^{-\sigma}\ln\rho$ holds for some constants $C_0,\sigma>0$, where $\sigma$ is the degree of the penalty function $\pi$ and $C_0$ in practice can be estimated using numerical results. Then  we  shall   define   the sets
\bb\l{eq:Gamma_rho}
\Gamma_{\rho,i}\coloneqq \{l\in \cN \mid |u^{\rho,i}_l-(\cM_iu^\rho)_l|\le C_0\rho^{-\sigma}
\ln \rho\}, \q  i\in \cI.
\ee
The next result demonstrates that $\Gamma_{\rho,i}$ in fact coincides with $\Gamma_{i}$ for large enough $\rho$.

\begin{Theorem}\l{thm:free-bdy}
Suppose that  there exist positive constants $C_0$ and $\sigma$ such that 
the estimate $0\le u-u^{\rho}\le C_0\rho^{-\sigma}\ln\rho$ holds for all $\rho>0$. For each $i\in \cI$, let $\Gamma_i$ and $\Gamma_{\rho,i}$ be the sets defined as in \eqref{eq:Gamma} and 
\eqref{eq:Gamma_rho}, respectively.
Then for a given switching cost $c> 0$, there exists $\rho_0>0$ such that $\Gamma_i=\Gamma_{\rho,i}$ for all $\rho\ge \rho_0$  and $i\in \cI$.
\end{Theorem}
\begin{proof}
We first show that $\Gamma_i\subset \Gamma_{\rho,i}$ for all  $\rho> 0$ and $i\in \cI$. For any fixed $i$, we can deduce from the estimate $u\le u^{\rho}+ C_0\rho^{-\sigma}\ln\rho$ and the monotonicity of $\cM_i$ that
$$
\cM_i u \le \cM_i (u^{\rho}+ C_0\rho^{-\sigma}\ln\rho)= \cM_i u^{\rho}+ C_0\rho^{-\sigma}\ln\rho.
$$
Now let $l$ be an arbitrary element of $\Gamma_i$ so that $u^i_l=(\cM_iu)_l$. It follows that 
\begin{align*}
u^{\rho,i}_l&\le u^{i}_l=(\cM_iu)_l\le \cM_i u^{\rho}+ C_0\rho^{-\sigma}\ln\rho,\\
u^{\rho,i}_l&\ge u^{i}_l- C_0\rho^{-\sigma}\ln\rho\ge (\cM_i u^{\rho})_l- C_0\rho^{-\sigma}\ln\rho,
\end{align*}
which implies that $l$ is in $\Gamma_{\rho,i}$.


Suppose the statement of Theorem \ref{thm:free-bdy} does not hold, then by using the finiteness of $\cN$ and the pigeonhole principle, there exists  a sequence $\{\rho_n\}$ such that $\rho_n\to \infty$ as $n\to\infty$, and an index $l\in \Gamma_{\rho_n,i}\setminus \Gamma_i$ for all $n$. However, the definition of $\Gamma_{\rho_n,i}$ implies that
\begin{align*}
u^i_l-(\cM_iu)_l&=u^i_l-u^{\rho_n,i}_l+u^{\rho_n,i}_l-(\cM_iu^{\rho_n})_l+(\cM_iu^{\rho_n})_l-(\cM_iu)_l\\
&\le C_0\rho_n^{-\sigma}\ln\rho_n+C_0\rho_n^{-\sigma}\ln\rho_n+0\to 0,
\end{align*}
which together with $u^i_l\ge (\cM_iu)_l$ leads to $l\in \Gamma_i$, and hence a contradiction.
\end{proof}

\section{Penalization errors for  vanishing switching cost}\l{sec:error_zero}

In this section, we investigate the asymptotic behaviours of Problems \ref{pro:switching} and \ref{pro:penalty} as the switching cost $c$ tends to zero. We shall show  that the system in Problem \ref{pro:switching} degenerates into a single equation of   HJB type, and establish that the penalty error of Problem \ref{pro:penalty} with zero switching cost is of the same order (up to a log factor) as that in Theorem \ref{thm:panalty_rate}. 

 Throughout this section, to emphasize the dependence on $c$, we shall denote by $u^{c}$ and $u^{c,\rho}$ the solutions to Problems \ref{pro:switching} and \ref{pro:penalty} with a given positive switching cost $c$, respectively, and by $u^{\rho}$ the solution to Problem \ref{pro:penalty} with $c=0$.
Moreover, to identify the limiting behaviour of Problems \ref{pro:switching} and \ref{pro:penalty}, we introduce the following regularity condition on the monotone system $F$:
\begin{Assumption}\l{assum:local}
The function $F$ in Problem \ref{pro:switching} is locally Lipschitz continuous.
\end{Assumption}

We emphasize that even though Assumption \ref{assum:concave} is a sufficient condition for Assumption \ref{assum:local}, it will not be used in this section. In particular, Assumption \ref{assum:local}  is  general enough to cover non-convex/non-concave equations, such as  Isaacs equations.

We first introduce the degenerate problem for zero switching cost.  
\begin{Problem} \l{pro:hjb}
Find $u\in \R^{N}$ such that the vector $\u=(u,\ldots, u)\in \R^{N\t d}$ satisfies
\bb\l{eq:hjb}
\min_{i\in \cI} F_i(\u)=0.
\ee
\end{Problem}

For the classical HJBQVIs where $F_i(u)=\inf_{\a\in\cA_i} B(\a)u^i-b(\a)$ for all $i\in\cI$, Problem \ref{pro:hjb} can be equivalently written as an HJB equation as studied in \cite{witte2011,witte2012}: find $u\in \R^{N}$ satisfying $\inf_{\a\in\cA} B(\a)u-b(\a)=0$ with $\cA=\cup_{i\in \cI}\cA_i$. However, we reiterate that in this work, $F_i$ can depend on all components of the system, and is not assumed to be concave in this section. Moreover, even  for  concave equations, both the penalty scheme, i.e., Problem \ref{pro:penalty} with $c=0$, and its error analysis are essentially different from those  in \cite{witte2011,witte2012}.

By using  the monotonicity condition \eqref{eq:mono}, one can easily establish the following comparison principle for Problem \ref{pro:hjb}, i.e., if $\u=(u,\ldots, u)\in \R^{N\t d}$ and $\v=(v,\ldots, v)\in \R^{N\t d}$ satisfy 
$ \min_{i\in \cI} F_i(\u)\le 0$ and $\min_{i\in \cI} F_i(\v)\ge 0$ respectively, then $u\le v$, which subsequently implies the uniqueness of solutions to  Problem \ref{pro:hjb}.  We shall now
 demonstrate that the solution to Problem \ref{pro:hjb} can be identified as  the limit of the solutions to Problem \ref{pro:switching} with vanishing switching cost. 
 
\begin{Proposition}\l{prop:switching_c}
Let $u^{c}$ solve  Problem \ref{pro:switching} with a switching cost $c>0$. Then  
 $(u^c)_{c>0}$  converges monotonically from below to 
the solution $\u$ of Problem \ref{pro:hjb} as $c\to 0$.  

\end{Proposition} 
\begin{proof}
It is easy to check that if $c_1>c_2>0$, then $u^{c_2}$ is a supersolution to Problem \ref{pro:switching} with a switching cost $c_1$. Hence the comparison principle and the a priori bound  (see Remark \ref{rmk:wp}) imply that for each $i\in\cI$, $(u^{c,i})_{c>0}$ converges monotonically from below to some vector $\bar{u}^i\in \R^{N}$ as $c\to 0$. Moreover, since $u^{c,i}\ge \cM_i u^c\ge u^{c,j}-c$ for all $j\not =i$, $c>0$, we have $\bar{u}^i\equiv \bar{u}$ for all $i\in\cI$.

We now show $\u\coloneqq (\bar{u},\ldots, \bar{u})$ solves Problem \ref{pro:hjb}. For any given $i\in\cI$, using the supersolution property of $u^c$, we have $F_i(u^c)\ge 0$ for all $c>0$. Hence  the continuity of $F$ implies that $\min_{i\in\cI}F_i(\u)\ge 0$. 
On the other hand, let $l\in\cN$ be a fixed index. For any given $c> 0$, we consider the component $i_{l,c}$ where $u^{c,i_{l,c}}_l=\max_{j\in\cI}u^{c,j}_l>(\cM_{i_{l,c}}u^c)_l$, and consequently $F_{i_{l,c}}(u^c)_l= 0$. As $c\to 0$, since $\cI$ is a finite set, by passing to a subsequence, we can assume there exists $\{c_n\}\to 0$ as $n\to \infty$, and a component $i_l\in\cI$ such that 
$F_{i_{l}}(u^{c_n})_l= 0$ for all $n\in\N$.
Thus letting $c_n\to 0$, we have
$\min_{i\in \cI} F_{i}(\u)_l \le F_{i_{l}}(\u)_l = 0$.
Since $l$ is an arbitrary index, we conclude $\min_{i\in \cI} F_{i}(\u)=0$.  
\end{proof}
 
Because Problem \ref{pro:hjb} is the limiting equation of Problem \ref{pro:switching} as $c\to 0$, we now analyze the approximation error of Problem \ref{pro:penalty} with $c=0$ to Problem \ref{pro:hjb}, which indicates the asymptotic behaviour of the penalization error of Problem \ref{pro:penalty} for small enough switching cost.

\begin{Theorem}\l{thm:rate_0}
The solution $u^\rho$ of Problem \ref{pro:penalty} (with $c=0$) converges monotonically from below to 
the solution $\u$ of Problem \ref{pro:hjb} as $\rho\to \infty$.
Moreover, if we further assume 
Assumption \ref{assum:local} holds and there exist positive constants $\tau$ and $\sigma$ such that $\pi(y)\ge \tau y^{1/\sigma}$ for all $0\le y\le 2\|F(0)\|/\gamma$, then 
 the following error estimate holds:
$$
0\le \u-u^\rho\le C_1/\rho^\sigma,
$$
for some constant $C_1>0$, independent of the penalty parameter $\rho$.
\end{Theorem} 
 \begin{proof}
 By Theorem \ref{thm:mono_conv}, $(u^\rho)_{\rho\ge 0}$ converge monotonically from below to some element $\u=(\bar{u}^1,\ldots, \bar{u}^d)\in \R^{N\t d}$. Since it holds that 
 $$
 0=\lim_{\rho\to\infty}\bigg\{F_i(u^{\rho})-\rho\sum_{j\not =i}\pi( u^{\rho,j}-u^{\rho,i})\bigg\}\le \lim_{\rho\to\infty}F_i(u^{\rho})=F_i(\u),
 $$
 we know $\bar{u}^1=\ldots=\bar{u}^d$ (otherwise the first limit above would blow up). It remains to establish that $\min_{i\in \cI}F_i(\u)\le 0$. To do so, we pick, for each $\rho>0$ and  $l\in \cN$, a component $i_{l,\rho}$ such that $u^{\rho,i_{l,\rho}}_l=\max_{j\in\cI}u^{\rho,j}_l$, so that $F_{i_{l,\rho }}(u^\rho)_l= 0$. The desired result is then established by passing to a subsequence as in the proof of Proposition \ref{prop:switching_c}.
 
 The fact that $u^\rho\le \u$ implies that   it suffices to show  there exists a constant $C_1$, such that for each $i\in \cI$ and $\rho> 0$,  $(u^{\rho,i},
 \ldots, u^{\rho,i})+C_1/\rho^\sigma\in \R^{N\t d}$ is a supersolution to  \eqref{eq:hjb}.
Note that $(u^\rho)_{\rho\ge 0}$ are  bounded by $\|F(0)\|/\gamma$ (see Lemma \ref{lem:bound_p}), hence we have 
 $$
 \sum_{j\not=i}\tau [(u^{\rho,j}-u^{\rho,i})^+]^{1/\sigma}\le \sum_{j\not=i}\pi(u^{\rho,j}-u^{\rho,i})\le C/\rho, \q \fa i\in \cI, \, j\not =i,
 $$
 where $C$ is defined as in \eqref{eq:C}. Then using the local Lipschitz continuity of  $F$, we obtain that
 $$
\|F_{j}(u^\rho) -F_j(\hat{\u}^i)\|\le L_{\textrm{lip}}\|u^\rho-\hat{\u}^i\| \le L_{\textrm{lip}} \left(\f{C}{\tau\rho}\right)^\sigma,
 $$
where $\hat{\u}^i\coloneqq (u^{\rho,i},\ldots, u^{\rho,i})$. Therefore, by using the inequality \eqref{eq:translation} and setting  $\gamma C_1=L_{\textrm{lip}} (C/\tau)^\sigma$, one can conclude $\hat{\u}^i+C_1/\rho^\sigma$ is a supersolution to  \eqref{eq:hjb} through the following estimate:
$$
F_j(\hat{\u}^i+C_1/\rho^\sigma)\ge F_j(\hat{\u}^i)+\gamma C_1/\rho^\sigma\ge F_{j}(u^\rho)\ge 0  \q \fa j\in \cI,
$$
which implies that $\hat{\u}^i+C_1/\rho^\sigma\ge \u$ for all $\rho> 0$ and $i\in \cI$.
 \end{proof}
\begin{Remark}\l{rmk:conv_c}
Under Assumption \ref{assum:local}, one can also show the rate of convergence for Problem \ref{pro:switching} to Problem \ref{pro:hjb} is of first order in the switching cost. In fact, let $u^{c}$ solve  Problem \ref{pro:switching} with a switching cost $c>0$. Then we have $\|u^{c,i}-u^{c,j}\|\le c$ for all $i,j\in \cI$. Then following the proof of Theorem \ref{thm:rate_0}, we see it holds for some constant $K>0$ that $0\le \u-u^c\le Kc$.
\end{Remark} 
 
Summarizing the above discussions, we can derive another  upper bound of the penalization error  for Problem \ref{pro:penalty} with the penalty function $\pi(y)=y^+$ and positive switching cost,  which enables us to explain the asymptotic behaviours of the penalty errors 
observed in Section \ref{sec:numerical}.
\begin{Theorem}\l{thm:sharper_bdd}
For any given switching cost $c> 0$ and  penalty parameter $\rho> 0$, let $u^{c}$ and $u^{c,\rho}$ be the solutions to Problems \ref{pro:switching} and \ref{pro:penalty}, respectively. Suppose Assumption \ref{assum:local} holds and the penalty function is given by $\pi(y)=y^+$. Then we have
$0\le u^{c}-u^{c,\rho}\le C_1(1/\rho+c\rho),$ for  some constant $C_1>0$, independent of $c$ and $\rho$.
\end{Theorem} 
\begin{proof}
Note that for any $b\in \R$ and $c>0$, we have
$$
c+\pi(b-c)=c+\max(b-c,0)=\max(b,c)\ge \max(b,0)=\pi(b),
$$
 which,  along with the inequality  \eqref{eq:translation} (with $u=u^{c,\rho}$ and $L=(d-1)c\rho/\gamma$), implies that  $u^{c,\rho}+(d-1)c\rho/\gamma$ is a supersolution to \eqref{eq:penalty} with $c=0$. Hence one can deduce from the comparison principle of the penalized equation \eqref{eq:penalty} the estimate $0\le u^{\rho}-u^{c,\rho}\le (d-1)c\rho/\gamma$. Then using  Proposition \ref{prop:switching_c} and Theorem \ref{thm:rate_0}, we conclude that
\begin{equation*}
u^{c}-u^{c,\rho}\le u^{c}-\u+\u-u^\rho+u^\rho-u^{c,\rho}\le C_1(1/\rho+c\rho),
\end{equation*}
where  $C_1$ is a constant independent of  $\rho$ and $c$.
\end{proof}

A direct consequence of   Theorems \ref{thm:panalty_rate} and \ref{thm:sharper_bdd} is that for the penalty function $\pi(y)=y^+$ and a fixed large enough penalty parameter, the following  estimate of the penalization error holds under Assumption \ref{assum:concave} as $c\to 0$:
\bb\l{eq:err_min}
0\le u^{c}-u^{c,\rho}\le C_1\min\bigg(-\f{\ln \rho}{\ln(1-c)\rho}, \; \f{1}{\rho}+c\rho\bigg),
\ee
for some constant $C_1$, independent of  $\rho$ and $c$. Thus for a sufficiently small switching cost, the penalty error is dominated by the term $C_1/\rho$, i.e., the penalization error with $c=0$.

\section{Numerical experiments}\l{sec:numerical}
 
In this section, we illustrate the theoretical findings and demonstrate the effectiveness of  the penalty schemes through numerical experiments. 
We present  an infinite-horizon optimal switching problem  and investigate the convergence of Problem \ref{pro:penalty} with respect to the penalty parameter. We shall also examine the dependence of the penalization errors on the switching cost.

Motivated by Remark \ref{rmk:penalty}, we shall focus on the penalty function with degree 1, i.e., $\pi(y)=y^+$.
Due to the semismoothness of the chosen penalty function, one can easily construct convergent iterative methods for solving Problem \ref{pro:penalty} with a fixed penalty parameter (see e.g.~\cite{bokanowski2009,witte2011,witte2012,reisinger2018} for details). Roughly speaking, starting with an initial guess $u^{(0)}$ of the solution to Problem \ref{pro:penalty}, for each $k\ge 0$, we compute the next iterate $u^{(k+1)}$ by solving 
$$
G^\rho[u^{(k)}]+\cL^{(k+1)}[u^{(k)}](u^{(k+1)}-u^{(k)})=0,
$$
where $\cL^{(k+1)}[u^{(k)}]$ is a generalized derivative of \eqref{eq:penalty} at the iterate $u^{(k)}$.  In practice, such a generalized derivative can be computed by policy iteration if $F_i(u)=\inf_{\a\in \cA_i}\sup_{\b\in \cB_i}B(\a,\b)u-b(\a,\b)$ for 
some sets $\cA_i, \cB_i$ and some coefficients $B$ and $b$, or more generally by a slanting function of $F=(F_i)_{i\in\cI}$ if it exists. One can further show that the iterates $(u^{k})_{k\ge 0}$ are  locally superlinearly convergent to the solution of Problem \ref{pro:penalty} or even globally convergent if $F$ is concave.

To motivate the discrete QVIs solved in our numerical experiments, we introduce the following infinite-horizon optimal switching problem (see e.g.~\cite{pham2009}). Let $(\Om, \cF_t, \bP)$ be a filtered probability space and $\a=(\a_t)_{t\ge 0}$ be a control process such that $\a_t=\sum_{k\ge 0}i_k1_{[\tau_k,\tau_{k+1})}(t)$, where  $(\tau_k)_{k\ge 0}$ is a non-decreasing sequence of stopping times representing the decision on ``when to switch", and 
for each $k\ge 0$, $i_k$ is an $\cF_{\tau_k}$-measurable random variable valued in the discrete space $\cI=\{1,\ldots, d\}$, $d\ge 2$, representing the   decision on ``where to switch". That is, the decision maker chooses  regime $i_k$ at the time $\tau_k$ for all $k\ge 0$. 

For any given control strategy $\a$, we consider the following controlled state  equation:
$$
dX^\a_t=(r+\nu(\a_t)(\mu-r))X^\a_t dt+\sigma\nu(\a_t) X^\a_t\,dB_t, \q t>0;\q X^\a_0=x,
$$
where $r,\mu,\sigma,x>0$ are given constants,  $(B_t)_{t>0}$ is a one-dimensional Brownian motion defined on $(\Om, \cF_t, \bP)$, and $\nu(i)=(i-1)/(d-1)$, $i\in \cI$. Then the objective function associated with the control strategy $\a$ is given by:
$$
J(x,\a)=\ex\bigg[\int_0^\infty e^{-rt}\ell(X^\a_t)\,dt-\sum_{k\ge 0}e^{-r\tau_{k+1}}c^{i_k,i_{k+1}}\bigg],
$$
where $\ell$ represents the running reward function and $c^{i,j}$ represents the switching cost from regime $i$ to $j$, $\fa i,j\in \cI$. For each $i\in \cI$, let $\bA^i$ be all  control strategies starting with regime $i$, i.e., $i_0=i$ and $\tau_0=0$. Then 
 the decision maker has the following value functions:
$$
V^i(x)=\sup_{\a\in \bA^i}J(x,\a), \q i\in \cI.
$$

Suppose the switching costs are positive, i.e., $c^{i,j}>0$ for $i\not =j$, then one can show by using the dynamic programming principle (see \cite{pham2009}) that the value function $V=(V^i)_{i\in \cI}$ satisfies the following system of quasi-variational inequalities: for all $i\in \cI$ and $x\in (0,\infty)$,
\bb\l{eq:qvi}
\min\bigg[-\f{1}{2}\sigma^2\nu(i)^2x^2 V^i_{xx}-(r+\nu(i)(\mu-r))xV^i_x+rV^i-\ell(x), V^i-(\cM_iV)\bigg]=0, 
\ee
where $\cM_iV=\max_{j\not =i}(V^i-c^{i,j})$. For our numerical tests, we  assume $c^{i,j}\equiv c$ for $i\not =j$, and set  other parameters  as $\sigma=0.2$, $\mu=0.06$, $r=0.02$.

Now we derive the finite-dimensional QVIs by discretizing \eqref{eq:qvi}. Note  that in this work, we focus on examining the  performance of penalty methods for solving  discrete QVIs resulting from discretizing \eqref{eq:qvi} with a fixed mesh size, instead of the convergence of the discretization to \eqref{eq:qvi} as the mesh size tends to zero. Therefore, for simplicity, we shall localize \eqref{eq:qvi} on the computational domain $(0,2)$ with homogenous Dirichlet boundary condition $u=0$ at $x=2$, and solve the localized equation on a uniform grid $\{x_l\}=\{lh\}_{l=0}^{N-1}$ with $h=2/N$.
We further derive a  monotone discretization of \eqref{eq:qvi}, which uses forward differences for the first derivates and central difference for all second derivatives. It is easy to verify that the resulting discrete system \eqref{eq:switching} satisfies the monotonicity condition \eqref{eq:mono} with $\gamma=0.02$ and also Assumption \ref{assum:concave}. 

We proceed to discuss implementation details for solving Problem \ref{pro:penalty} with semismooth Newton methods. 
The initial guess $u^{(0)}$ shall be taken as the solution to  \eqref{eq:penalty} with $\rho=0$, i.e., $F_i(u)=0$ for all $i\in \cI$, and the iterations will be terminated once  the desired tolerance is achieved, i.e., 
$\f{\|u^{(k)}-u^{(k-1)}\|}{\max(\|u^{(k)}\|,\textrm{scale})}<\textrm{tol}$, where the scale parameter is chosen to guarantee that no unrealistic level of accuracy will   be imposed if the solution is close to zero. 
We take $\textrm{tol}=10^{-9}$ and $\textrm{scale}=1$ for all experiments. Computations are performed using \textsc{Matlab} R2016a on a  laptop with 2.2 GHz Intel Core i7 and 16 GB memory.\footnotemark

\footnotetext{The \textsc{Matlab} code of the numerical experiments can be found via the link: https://github.com/yfzhang01/Penalty-methods-for-optimal-switchings.git.}

 \begin{table}[!t]
 \renewcommand{\arraystretch}{1.05} 
\centering
\begin{tabular}[t]{@{}rcccccc@{}}
\toprule

$\rho$ &$10^3$& $2\t10^3$ & $4\t10^3$ & $8\t10^3$ & $16\t10^3$ & $32\t10^3$\\
\midrule
$c=1/2$\\
 (a) & 3.37521 &3.38261	&3.38633	&3.38819	&3.38913	&3.38959	\\
 (b)  & &	0.00884	&0.00444	&0.00222	&0.00111	&0.00056	 \\ 
 (c)  & 5	&6	&6	&6	&6	&6	\\ 
 (d)  &0.0021	&0.0026	&0.0027	&0.0034	&0.0031	&0.0027	\\
$c=1/8$\\
 (a) & 5.26287 &5.27999	&5.28860	&5.29292	&5.29508	&5.29617	\\ 
 (b)  & &	0.02039	&0.01025	&0.00514	&0.00258	&0.00129	 \\ 
 (c)  & 7	&5	&5	&5	&5	&5	\\ 
 (d)  &0.0041	&0.0026	&0.0025	&0.0025	&0.0025	&0.0024	\\
$c=1/32$\\
 (a) & 5.98193 &6.01704	&6.03478	&6.04370	&6.04817	&6.05041	\\
 (b)  & &	0.04183	&0.02114	&0.01063	&0.00533	&0.00267	 \\ 
 (c)  & 6	&6	&5	&5	&5	&5	\\ 
 (d)  &0.0038	&0.0036	&0.0032	&0.0033	&0.0029	&0.0024	\\
$c=1/128$\\
 (a) & 6.23801 &6.30708	&6.34232	&6.36011	&6.36906	&6.37354	\\
 (b)  & &	0.08234	&0.04201	&0.02122	&0.01066	&0.00534	 \\
 (c)  & 5	&5	&4	&4	&4	&4	\\ 
 (d)  &0.0022	&0.0026	&0.0023	&0.0019	&0.0019	&0.0018	\\
$c=1/512$\\
 (a) & 6.35128 &6.42179	&6.45776	&6.47593	&6.48506	&6.48964	\\
 (b)  & &	0.08406	&0.04288	&0.02166	&0.01089	&0.00546	 \\ 
 (c)  & 5	&5	&4	&4	&4	&4	\\ 
 (d)  &0.0021	&0.0022	&0.0018	&0.0018	&0.0019	&0.0018	\\
$c=1/2048$\\
 (a) & 6.37959 &6.45047	&6.48662	&6.50488	&6.51406	&6.51866	\\
 (b)  & &	0.08449	&0.04310	&0.02177	&0.01094	&0.00548	 \\ 
 (c)  & 4	&4	&4	&4	&4	&4	\\ 
 (d)  &0.0018	&0.0019	&0.0019	&0.0018	&0.0019	&0.0019	\\
$c=0$\\
 (a) & 6.38903 &6.46003	&6.49624	&6.51454	&6.52373	&6.52834	\\
 (b)  & &	0.08464	&0.04318	&0.02181	&0.01096	&0.00549	 \\ 
 (c)  & 4	&4	&3	&3	&3	&3	\\ 
 (d)  &0.0018	&0.0019	&0.0015	&0.0015	&0.0015	&0.0015	\\
\bottomrule
\end{tabular}
\caption{Numerical results for the two-regime optimal switching problem with different switching costs and penalty parameters. Shown are: (a) the numerical solutions $u^{c,\rho,1}$ at $x=0.5$; (b) the increments $\|u^{c,\rho} -u^{c,\rho/2}\|$; (c) the number of iterations; (d) the overall runtime in seconds.}
\label{table:2d}
\end{table}%

We first study the performance of the penalty approximation for the discrete  system corresponding to the two-regime case, i.e., $d=2$, and the running reward $\ell(x)=2(1-x)1_{(0.75,1]}(x)$. Note that the discontinuity of $\ell$ at $x=0.75$ should not affect our convergence analysis since we are solving a finite-dimensional nonlinear system resulting from a fixed discretization of \eqref{eq:qvi}.

Table \ref{table:2d} contains, for different switching costs and penalty parameters, the numerical solutions of Problem \ref{pro:penalty} with a fixed mesh size $h=0.02$ (hence the total number of unknowns is $2N=200$). Line (a) shows that for a fixed switching cost $c$, regardless of whether $c$ is positive or not, the numerical solutions converge monotonically from below to the exact solution as the penalty parameter $\rho\to \infty$. The first-order convergence of the penalization error (in the sup-norm) with respect to the penalty parameter $\rho$ can be deduced from line (b), which confirms the theoretical  results (the log factor has not been observed, 
c.f.~Theorems \ref{thm:panalty_rate} and \ref{thm:rate_0}). 
Moreover,  by fixing the penalty parameter $\rho$ and comparing the increment $\|u^{c,\rho} -u^{c,\rho/2}\|$ columnwise,
 one can observe that the penalty errors first grow  at a rate $1/2$ when the switching cost decreases from $1/2$ to $1/128$, and then stablize to the penalty errors of the  limiting case (with $c=0$)  when $c$ tends to $0$, as asserted by \eqref{eq:err_min}.

The lines (c) and (d) clearly indicate the efficiency of the iterative solver.
We remark that compared with parabolic QVIs, elliptic QVIs are more challenging to solve due to the fact we cannot take the solution at  the previous timestep as an accurate initial guess \cite{forsyth2002,azimzadeh2016weakly}. In fact, Figure \ref{fig:initial} (left) illustrates a large disagreement in the shape and magnitude between  the initial guess $u^{(0)}$ and the final  solution $u^{c,\rho}$ of  Problem \ref{pro:penalty} with $\rho=10^3$ and $c=1/8$. However, 
 we can see that the iterative method solves Problem \ref{pro:penalty} at the accuracy $10^{-9}$ using only a small number of iterations within several milliseconds, which seems to be independent of the size of the penalty parameter $\rho$.

\begin{figure}[htbp]
    \centering
    \includegraphics[width=0.42\columnwidth,height=5cm]{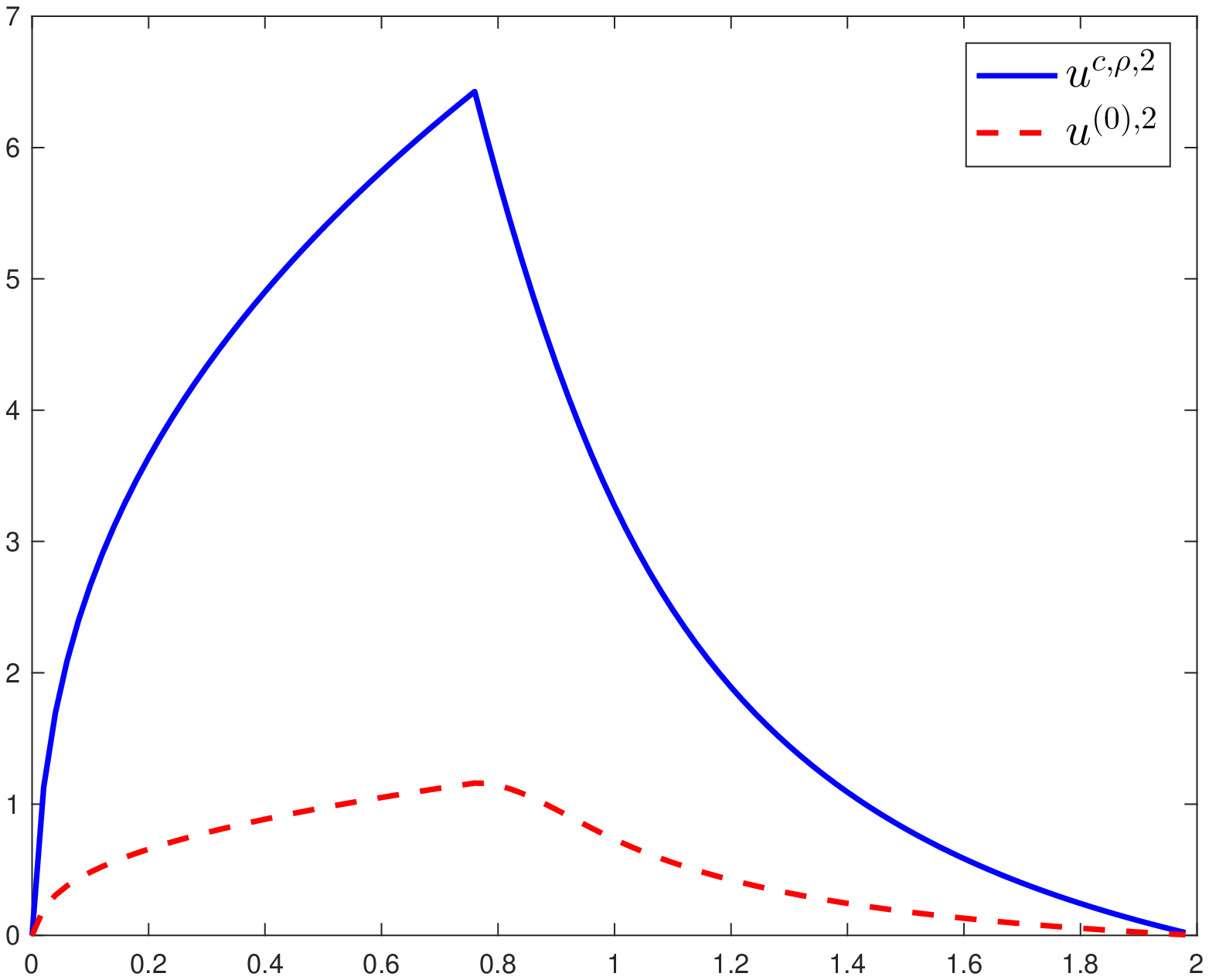}\hfill
    \includegraphics[width=0.42\columnwidth,height=5cm]{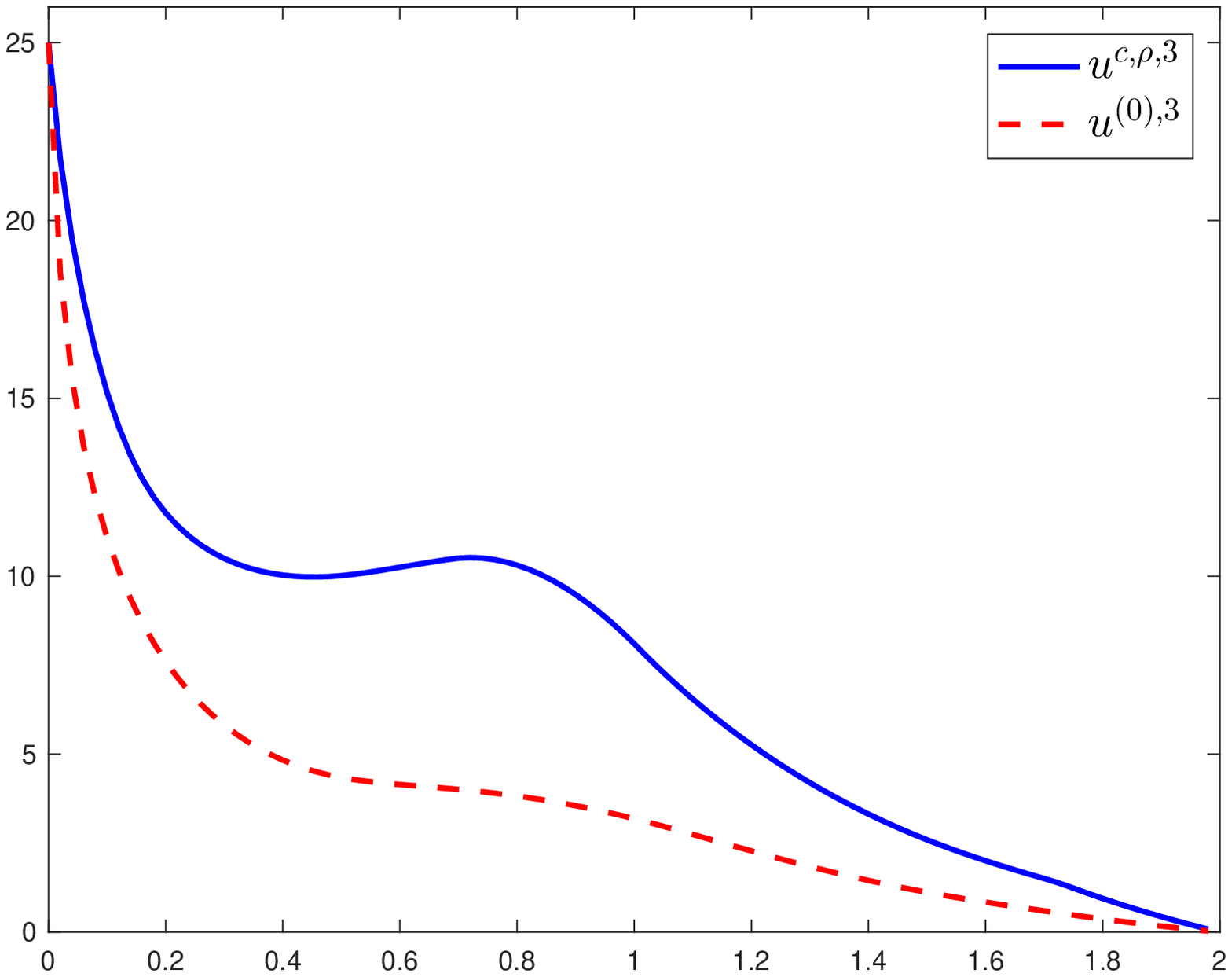}
    \caption{Differences between the last components of the initial guess $u^{(0)}$ and the   solution $u^{c,\rho}$ for  the penalized switching system. Shown are: two-regime problem with $\rho=10^3$ and $c=1/8$ (left), and three-regime problem with $\rho=4\t 10^3$ and $c=1/1024$  (right).}
    \label{fig:initial}
\end{figure}

We now turn to analyze  the convergence of the penalty methods for the nonlinear  system resulting from a three-regime problem, i.e., $d=3$. The running reward function is chosen as 
$$
\ell(x)=\begin{cases} -(x-0.5), & x\in (0,0.5],\\ x-0.5, &x\in (0.5,1], \\ -(x-1.5),  &x\in (1,1.5], \\ x-1.5, &x\in (1.5,1.75],\\
0, &\textnormal{otherwise},
\end{cases}
$$
which admits a mixed convexity. Table \ref{table:3d} presents the numerical solutions of the three-regime penalized equations with   a fixed   mesh $h=0.02$ (hence the total number of unknowns is $3N=300$), and different penalty parameter $\rho$ and switching cost $c$.
Lines (a) and (b) indicate that for a fixed switching cost, the numerical solutions converge monotonically with a first-order accuracy, as the penalty parameter tends to infinity.
Moreover, similar to the two-regime problem, we can  observe that as the switching cost tends to zero, the penalty errors corresponding to a fixed penalty parameter $\rho$ first increase at a rate  $\cO(c^{-1/2})$ (for $c\in [1/4,1/1024]$) and then approach to the penalization errors with $c=0$. Lines (c) and (d) summarize the number of required iterations and the computational time,  which illustrate the efficiency of the iterative solvers for the penalized problems. Despite the relatively poor initial guess (c.f.~Figure \ref{fig:initial} (right)),   the desired accuracy $10^{-9}$ is in general obtained  within 0.015 seconds using a reasonable amount of iterations, which does not depend on the magnitude of the penalty parameter.

 \begin{table}[!t]
 \renewcommand{\arraystretch}{1.05} 
\centering
\begin{tabular}[t]{@{}rcccccc@{}}
\toprule

$\rho$ &$4\t10^3$ & $8\t10^3$ & $16\t10^3$ & $32\t10^3$ & $64\t 10^3$& $128\t10^3$\\
\midrule
$c=1/4$\\
 (a) &6.849917&	6.849942&	6.849954&	6.849960&	6.849962&	6.849964	 \\
 (b)  & &	0.000208&	0.000104&	0.000052&	0.000026&	0.000013	 \\ 
 (c)  & 12&	12&	12&	12&	12&	12	\\ 
 (d)  &0.0112&	0.0113&	0.0111&	0.0112&	0.0113&	0.0113	\\
$c=1/16$\\
 (a) &7.405239&	7.405507&	7.405641&	7.405708&	7.405742&	7.405758	 \\
 (b)  & &	0.000451&	0.000226&	0.000113&	0.000056&	0.000028	 \\ 
 (c)  & 12&	12&	12&	12&	12&	12	\\ 
 (d)  &0.0115&	0.0115&	0.0115&	0.0111&	0.0115&	0.0117	\\
$c=1/64$\\
 (a) &7.791271&	7.792091&	7.792499&	7.792703&	7.792805&	7.792856	 \\
 (b)  & &	0.001003&	0.000501&	0.000250&	0.000125&	0.000062	 \\ 
 (c)  & 13&	13&	13&	13&	13&	13	\\ 
 (d)  &0.0121&	0.0122&	0.0151&	0.0130&	0.0127&	0.0127	\\
$c=1/256$\\
 (a) &8.009477&	8.011330&	8.012258&	8.012722&	8.012955&	8.013071	 \\
 (b)  & &	0.002016&	0.001010&	0.000505&	0.000253&	0.000126	 \\
 (c)  & 14&	14&	14&	14&	14&	14	\\ 
 (d)  &0.0131&	0.0130&	0.0132&	0.0129&	0.0129&	0.0130	\\
$c=1/1024$\\
 (a) &8.108554&	8.112341&	8.114262&	8.115229&	8.115715&	8.115958	 \\
 (b)  & &	0.003980&	0.002018&	0.001017&	0.000510&	0.000256	 \\ 
 (c)  & 15&	15&	14&	15&	15&	15	\\ 
 (d)  &0.0144&	0.0145&	0.0130&	0.0145&	0.0144&	0.0143	\\
$c=1/4096$\\
 (a) &8.135298&	8.138958&	8.141012&	8.142047&	8.142567&	8.142828	 \\
 (b)  & &	0.003854&	0.002156&	0.001087&	0.000546&	0.000273	 \\ 
 (c)  & 14&	14&	14&	14&	14&	14	\\ 
 (d)  &0.0135&	0.0130&	0.0131&	0.0132&	0.0133&	0.0132	\\
$c=1/16384$\\
 (a) &8.143553&	8.146389&	8.147826&	8.148752&	8.149280&	8.149545	 \\
 (b)  & &	0.002975&	0.001508&	0.000974&	0.000554&	0.000278	 \\ 
 (c)  & 12&	12&	14&	14&	14&	14	\\ 
 (d)  &0.0115&	0.0115&	0.0134&	0.0134&	0.0134&	0.0133	\\
$c=0$\\
 (a) &8.146313&	8.149164&	8.150603&	8.151326&	8.151688&	8.151869	 \\
 (b)  & &	0.002990&	0.001509&	0.000758&	0.000380&	0.000190 \\ 
 (c)  & 12&	12&	12&	12&	12&	11\\ 
 (d)  &0.0112&	0.0112&	0.0111&	0.0111&	0.0111&	0.0103	\\
\bottomrule
\end{tabular}
\caption{Numerical results for the three-regime optimal switching problem with different switching costs and penalty parameters. Shown are: (a) the numerical solutions $u^{c,\rho,1}$ at $x=1$; (b) the increments $\|u^{c,\rho} -u^{c,\rho/2}\|$; (c) the number of iterations; (d) the overall runtime in seconds.}
\label{table:3d}
\end{table}%

\section{Conclusions}
In this paper, we show that the penalty method is a powerful tool to solve a large class of discrete quasi-variational inequalities arising from hybrid  control problems involving  switching controls. We   establish  monotone convergence for the solutions of the penalized equations related to a general class of penalty functions, and rigorously analyze the penalization errors  for both positive switching cost and zero switching cost. These error estimates  further lead to an exact construction of the optimal switching regions. Numerical examples for infinite-horizon optimal switching problems are presented to illustrate the  theoretical findings.

To the best of our knowledge, this is the first paper which proposes penalty approximations for QVIs in such a generality and presents  rigorous error estimates for the penalization errors. Natural next steps would be to extend the penalty approach to interconnected  obstacles with negative switching costs as in \cite{pham2009}, to more general intervention operators as in \cite{azimzadeh2016weakly}, and to  monotone systems with interconnected bilateral obstacles as in \cite{djehiche2017}.

\clearpage
%
%


\begin{thebibliography}{1}

\bibitem{azimazadeh2018}
{\sc P.~Azimzadeh, E.~Bayraktar, and G.~Labahn}, \emph{Convergence of implicit schemes for Hamilton-Jacobi-Bellman quasi-variational inequalities}, SIAM J. Control Optim., 56 (2018), pp.~3994--4016.


\bibitem{azimzadeh2016weakly}
{\sc P.~Azimzadeh and P.~A.~Forsyth}, \emph{Weakly chained matrices, policy iteration, and
impulse control}, SIAM J. Numer. Anal., 54 (2016), pp.~1341--1364.



\bibitem{babbin2014}
{\sc J.~Babbin, P.~A.~Forsyth and G.~Labahn}, \emph{A comparison of iterated optimal stopping and
local policy iteration for American options under regime switching}, J. Sci. Comput., 58
(2014), pp.~409--430.

\bibitem{bensoussan1982}
{\sc A.~Bensoussan and J.~L.~Lions}, \emph{Contr\^{o}le impulsionnel et in\'{e}quations quasi variationnelles},
vol. 11 of M\'{e}thodes Math\'{e}matiques de l'Informatique [Mathematical Methods
of Information Science], Gauthier-Villars, Paris, 1982.



\bibitem{bokanowski2009}
 {\sc O. Bokanowski, S. Maroso and H. Zidani}, \emph{Some convergence results for Howard's algorithm}, SIAM J. Numer. Anal., 47 (2009), pp.~3001--3026.

\bibitem{bonnans2007}
{\sc F. Bonnans, S. Maroso and H. Zidani}, \emph{Error estimates for a stochastic impulse
control problem}, Appl. Math. Optim., 55 (2007), pp.~327--357.

\bibitem{boulbrachene2001fem}
{\sc M.~Boulbrachene and M.~Haiour}, \emph{The finite element approximation of Hamilton-Jacobi-Bellman
equations}, Comput. Math. Applic., 41 (2001),~993--1007.

\bibitem{briani2012}
{\sc A.~Briani,  F.~Camilli, H.~Zidani}, \emph{Approximation schemes for monotone systems of nonlinear second
order differential equations: convergence result and error estimate}, Diff. Equ. Appl., 4 (2012), pp.~297--317.


\bibitem{chen2000}
{\sc X. Chen, Z. Nashed and L. Qi}, \emph{Smoothing methods and semismooth methods for nondifferentiable operator equations}, SIAM J. Numer. Anal., 38 (2000), pp.~1200--1216.

\bibitem{djehiche2017}
{\sc B.~Djehiche, S.~Hamadene, M.~A.~Morlais and X.~Zhao}, \emph{On the equality of solutions of max-min and
min-max systems of variational inequalities with interconnected bilateral obstacles}. J. Math. Anal. Appl.,
452 (2017), pp.~148--175.

\bibitem{dockner2000}
{\sc E.~Dockner, S.~Jorgensen, N.~Van Long and G.~Sorger}, \emph{Differential Games in Economics
and Management Science}, Cambridge University Press, Cambridge, UK, 2000.

\bibitem{ferretti2017}
{\sc R.~Ferretti, A.~Sassi and H.~Zidani}, \emph{Error estimates for numerical approximation of Hamilton-Jacobi equations related to hybrid control systems}, Appl. Math. Optim., to appear.



\bibitem{forsyth2002}
{\sc P.~A.~Forsyth and K.~R.~Vetzal}, \emph{Quadratic convergence for valuing American options using
a penalty method}, SIAM J. Sci. Comput., 23 (2002), pp.~2095--2122.

\bibitem{huang2010}
{\sc C.~C.~Huang and S.~Wang}, \emph{A power penalty approach to a nonlinear complementarity problem}, Oper. Res.
Lett., 38 (2010), pp.~72--76.

\bibitem{ishii1993}
{\sc K.~Ishii}, \emph{Viscosity solutions of nonlinear second order elliptic PDEs associated with impulse
control problems}, Funkcial. Ekvac., 36 (1993), pp.~123--141.

\bibitem{ito2006}
{\sc K. Ito and K. Kunisch}, \emph{Parabolic variational inequalities: The Lagrange multiplier approach}, J. Math. Pures Appl., 85 (2006), pp.~415--449.

\bibitem{ito2001}
{\sc K.~Ito and J.~Zou}, \emph{Identification of some source densities of the distribution type}, J.
Comput. Appl. Math., 132 (2001), pp.~295--308.

\bibitem{jakobsen2006}
{\sc E.~R.~Jakobsen}, 
\emph{On error bounds for monotone approximation schemes for multi-dimensional Isaacs equations}, Asymptot. Anal., 49 (2006), pp. 249--273.

\bibitem{kharroubi2010}
{\sc I.~Kharroubi, J.~Ma, H.~Pham, and J.~Zhang}, \emph{Backward SDEs with constrained jumps and quasi-variational inequalities}, Ann. Probab., 38 (2010), pp.~794--840.

\bibitem{ortega2000}
{\sc J.~M.~Ortega and W.~C.~Rheinboldt}, \emph{Iterative Solution of Nonlinear Equations in Several Variables},
Classics Appl. Math. 30, SIAM, Philadelphia, 2000.

\bibitem{pham2009}
{\sc H.~Pham}, \emph{Continuous-time Stochastic Control and Optimization with Financial Applications},
Stoch. Model. Appl. Probab. 61, Springer Verlag, Berlin, 2009.

\bibitem{reisinger2018}
{\sc C.~Reisinger and Y.~Zhang}, \emph{A penalty scheme and policy iteration for nonlocal HJB variational inequalities with monotone drivers}, preprint, arXiv:1805.06255 [math.NA], 2018.

\bibitem{seydel2009}
{\sc R.~C.~Seydel}, \emph{Impulse control for jump-diffusions: viscosity solutions of quasi-variational inequalities
and applications in bank risk management}, PhD Thesis, Leipzig University, 2009.


\bibitem{witte2011}
{\sc J.~H.~Witte and C.~Reisinger}, 
\emph{A penalty method for the numerical solution of Hamilton-Jacobi-
Bellman (HJB) equations in finance}, SIAM J. Numer. Anal., 49 (2011), pp.~213--231.

\bibitem{witte2012}
{\sc J.~H.~Witte and C.~Reisinger}, \emph{Penalty methods for the solution of discrete HJB equations:
Continuous control and obstacle problems}, SIAM J. Numer. Anal., 50 (2012), pp.~595--625.






\end{thebibliography}
\end{document}